\documentclass[11 pt]{article}
\title{A Birman exact sequence for the Torelli subgroup of $\Aut(F_n)$}
\author{Matthew Day\thanks{Supported in part by NSF grant DMS-1206981}\ \ and Andrew Putman\thanks{Supported in part by NSF grant DMS-1255350 and the Alfred P.\ Sloan Foundation}}
\date{March 30, 2016}

\newcommand{\arxiv}[1]{\href{http://arxiv.org/abs/#1}{{\tt arXiv:#1}}}

\usepackage{mathtools}
\usepackage{amsmath,amssymb,amsthm,amscd,amsfonts}
\usepackage{paralist}
\usepackage[hmargin=1.25in, vmargin=1.185in]{geometry}
\usepackage[font=small,format=plain,labelfont=bf,up,textfont=it,up]{caption}
\usepackage{type1cm}
\usepackage{calc}
\usepackage{enumerate}
\usepackage{url}
\usepackage{bm}
\usepackage{booktabs}
\usepackage{microtype}
\usepackage{titling}
\usepackage[all,cmtip]{xy}
\usepackage{mathabx}
\usepackage{bbm}

\usepackage[dvipdfm, bookmarks, bookmarksdepth=2, colorlinks=true, linkcolor=blue, citecolor=blue, urlcolor=blue]{hyperref}

\theoremstyle{plain}
\newtheorem{theorem}{Theorem}[section]
\newtheorem{maintheorem}{Theorem}

\newtheorem{lemma}[theorem]{Lemma}

\newtheorem{claimsa}{Claim}
\newtheorem{claimsb}{Claim}

\theoremstyle{definition}

\theoremstyle{remark}
\newtheorem{remark}[theorem]{Remark}

\newtheorem*{example}{Example}

\DeclareMathOperator{\Hom}{Hom}

\DeclareMathOperator{\IA}{IA}

\DeclareMathOperator{\GL}{GL}

\newcommand\Z{\ensuremath{\mathbb{Z}}}
\newcommand\Q{\ensuremath{\mathbb{Q}}}

\DeclareMathOperator{\HH}{H}

\DeclareMathOperator{\Aut}{Aut}
\DeclareMathOperator{\End}{End}
\DeclareMathOperator{\Out}{Out}

\newcommand\Set[2]{\ensuremath{\{\text{#1 $|$ #2}\}}}
\newcommand\Pres[2]{\ensuremath{\langle \text{$#1$ $|$ $#2$} \rangle}}
\newcommand\LPres[3]{\ensuremath{\langle \text{$#1$ $|$ $#2$ $|$ $#3$} \rangle}}

\newcommand\AutFB[2]{\ensuremath{\mathcal{A}_{#1,#2}}}
\newcommand\IAB[2]{\ensuremath{\mathcal{A}_{#1,#2}^{\text{IA}}}}
\newcommand\BKer[2]{\ensuremath{\mathcal{K}_{#1,#2}}}
\newcommand\BKerIA[2]{\ensuremath{\mathcal{K}^{\text{IA}}_{#1,#2}}}
\newcommand\One{\ensuremath{\mathbbm{1}}}

\newcommand\GroupPres[2]{\ensuremath{\langle \text{$#1$ $|$ $#2$} \rangle}}

\newcommand\Conj[1]{\ensuremath{\ldbrack #1 \rdbrack}}
\newcommand\Mul[2]{\ensuremath{M_{#1,#2}}}
\newcommand\Mulcomm[3]{\ensuremath{M_{#1,[#2,#3]}}}
\newcommand\Con[2]{\ensuremath{C_{#1,#2}}}
\newcommand\Swap[2]{\ensuremath{P_{#1,#2}}}
\newcommand\Inv[1]{\ensuremath{I_{#1}}}

\newcommand\kMul[2]{\ensuremath{\mathfrak{M}_{#1,#2}}}
\newcommand\kMulcomm[3]{\ensuremath{\mathfrak{M}_{#1,[#2,#3]}}}
\newcommand\kCon[2]{\ensuremath{\mathfrak{C}_{#1,#2}}}
\newcommand\kSwap[2]{\ensuremath{\mathfrak{P}_{#1,#2}}}
\newcommand\kInv[1]{\ensuremath{\mathfrak{I}_{#1}}}

\newcommand\hf{\ensuremath{\widehat{f}}}
\newcommand\hs{\ensuremath{\widehat{s}}}
\newcommand\hw{\ensuremath{\widehat{w}}}
\newcommand\hu{\ensuremath{\widehat{u}}}
\newcommand\hv{\ensuremath{\widehat{v}}}

\newcommand\co{\colon}
\newcommand{\fm}[1]{{#1}^*}
\newcommand{\zasd}{\Z^n\rtimes\Aut(F_n)}
\newcommand{\zasdr}{\Z^n\rtimes_r \Aut(F_n)}

\newcommand\AB[2]{\ensuremath{\prescript{#1}{}{#2}}}
\newcommand\AK[2]{\ensuremath{\alpha_{#1}\left(#2\right)}}
\newcommand\BK[2]{\ensuremath{\beta_{#1}\left(#2\right)}}

\newcommand\tlambda{\ensuremath{\widetilde{\lambda}}}

\begin{document}

\maketitle

\begin{abstract}
We develop an analogue of the Birman exact sequence for the Torelli subgroup of $\Aut(F_n)$.
This builds on earlier work of the authors who studied an analogue of the
Birman exact sequence for the entire group $\Aut(F_n)$.  These results play
an important role in the authors' recent work on the second homology group of the Torelli group.
\end{abstract}

\section{Introduction}

The Birman exact sequence \cite{BirmanSeq, FarbMargalitPrimer} is a fundamental result
that relates the mapping class groups of surfaces with differing numbers of boundary components.
It is frequently used to understand the stabilizers in the mapping class group of
simple closed curves on a surface.  In \cite{DayPutmanBirman}, the authors constructed
an analogous exact sequence for the automorphism group $\Aut(F_n)$ of the free group $F_n$
on $n$ letters $\{x_1,\ldots,x_n\}$.  The {\em Torelli subgroup} of $\Aut(F_n)$, denoted $\IA_n$, is the kernel
of the map $\Aut(F_n) \rightarrow \GL_n(\Z)$ obtained from the action of $\Aut(F_n)$ on
$F_n^{\text{ab}} \cong \Z^n$.  In this paper, we construct a version of the Birman exact
sequence for $\IA_n$.  This new exact sequence plays a key role in our recent paper \cite{DayPutmanH2}
on $\HH_2(\IA_n;\Z)$.

\paragraph{Birman exact sequence.}
Let $F_{n,k}$ be the free group on the set $\{x_1,\ldots,x_n,y_1,\ldots,y_k\}$.  For $z \in F_{n,k}$, let
$\Conj{z}$ denote the conjugacy class of $z$.  Define
\[\AutFB{n}{k} = \Set{$f \in \Aut(F_{n,k})$}{$\Conj{f(y_i)} = \Conj{y_i}$ for $1 \leq i \leq k$}.\]
The map $F_{n,k} \rightarrow F_n$ whose kernel is the normal closure of $\{y_1,\ldots,y_k\}$ induces
a map $\pi\co\AutFB{n}{k} \rightarrow \Aut(F_n)$.  The inclusion $\Aut(F_n) \hookrightarrow \AutFB{n}{k}$
whose image consists of automorphisms that fix the $y_i$ pointwise is a right inverse for $\pi$, so
$\pi$ is a split surjection.  Let $\BKer{n}{k} = \ker(\pi)$, so we have a split short exact sequence
\begin{equation}
\label{eqn:besaut}
1 \longrightarrow \BKer{n}{k} \longrightarrow \AutFB{n}{k} \stackrel{\pi}{\longrightarrow} \Aut(F_n) \longrightarrow 1.
\end{equation}
This is the Birman exact sequence for $\Aut(F_n)$ that was studied in \cite{DayPutmanBirman}.
In that paper, the authors
proved that $\BKer{n}{k}$ is finitely generated but not finitely presentable, constructed a simple
infinite presentation for it, and computed its abelianization.  
We
say more about these results below.

\begin{remark}
In \cite{DayPutmanBirman}, slightly more general groups $\mathcal{A}_{n,k,l}$ and $\mathcal{K}_{n,k,l}$ were studied.
To simplify our exposition, we decided to focus on the case $l=0$ in this paper.
\end{remark}

\paragraph{Analogue for Torelli.}
Define $\IA_{n,k}$ to be the Torelli subgroup of $\Aut(F_{n,k})$.  Set
$\IAB{n}{k} = \AutFB{n}{k} \cap \IA_{n,k}$ and $\BKerIA{n}{k} = \BKer{n}{k} \cap \IA_{n,k}$.  The
exact sequence \eqref{eqn:besaut} restricts to a split short exact sequence
\begin{equation}
\label{eqn:besia}
1 \longrightarrow \BKerIA{n}{k} \longrightarrow \IAB{n}{k} \longrightarrow \IA_n \longrightarrow 1.
\end{equation}
This is our Birman exact sequence for $\IA_n$.  The purpose of this paper is to prove
results for $\BKerIA{n}{k}$ that are analogous to the results for $\BKer{n}{k}$ that we listed above.

\paragraph{Comparing the kernels.}
The group $\BKerIA{n}{k}$ is the kernel of the restriction of the map $\Aut(F_{n,k}) \rightarrow \GL_n(\Z)$
to $\BKer{n}{k}$.  The image of this restriction is isomorphic to $\Z^{nk}$.  Indeed, using the
generating set for $\BKer{n}{k}$ constructed in \cite{DayPutmanBirman} (see below), one can show that with respect
to the basis $\{[x_1],\ldots,[x_n],[y_1],\ldots,[x_k]\}$
for $\Z^{n+k}$, it consists of matrices of the form
\[\left(\begin{matrix} \One_n & 0 \\ A & \One_k \end{matrix}\right),\]
where $\One_n$ and $\One_k$ are the $n \times n$ and $k \times k$ identity matrices and 
$A$ is an arbitrary $k \times n$ integer matrix.  We thus
have a short exact sequence
\begin{equation}
\label{eqn:bkeriabker}
1 \longrightarrow \BKerIA{n}{k} \longrightarrow \BKer{n}{k} \longrightarrow \Z^{nk} \longrightarrow 1.
\end{equation}
Unfortunately, 
it is difficult to use this exact sequence to deduce results about the combinatorial
group theory of $\BKerIA{n}{k}$ from analogous results for $\BKer{n}{k}$ (although in a sense we do this in the proof of Theorem~\ref{maintheorem:notfp} below).
For instance, 
the authors proved in \cite{DayPutmanBirman} that $\BKer{n}{k}$ is finitely generated,
but this does not directly imply anything about generating sets for $\BKerIA{n}{k}$.

\paragraph{Generators.}
We now turn to our theorems.  Set $X = \{x_1,\ldots,x_n\}$ and $Y = \{y_1,\ldots,y_k\}$.
For distinct $z,z' \in X \cup Y$, define $\Con{z}{z'} \in \Aut(F_{n,k})$
via the formula
\[\Con{z}{z'}(s) = \begin{cases}
z's(z')^{-1} & \text{if $s=z$},\\
s & \text{otherwise}.\end{cases}
\quad \quad (s \in X \cup Y).\]
Also, for $z \in X \cup Y$ and $\alpha = \pm 1$ and $v \in F_{n,k}$ in
the subgroup generated by $(X \cup Y) \setminus \{z\}$, define
$\Mul{z^{\alpha}}{v} \in \Aut(F_{n,k})$ via the formula
\[\Mul{z^{\alpha}}{v}(s) = \begin{cases}
v s & \text{if $s = z$ and $\alpha = 1$},\\
s v^{-1} & \text{if $s = z$ and $\alpha = -1$},\\
s & \text{otherwise.}\end{cases}
\quad \quad (s \in X \cup Y).\]
Observe that with this definition we have
$\Mul{z^{\alpha}}{v}(z^{\alpha}) = v z^{\alpha}$.
The authors proved in \cite{DayPutmanBirman} that $\BKer{n}{k}$ is generated by the finite set
\begin{equation}
\label{eqn:bkergen}
\Set{$\Mul{x}{y}$}{$x \in X$, $y \in Y$} \cup
\Set{$\Con{y}{z}$, $\Con{z}{y}$}{$y \in Y$, $z \in (X \cup Y) \setminus \{y\}$}.
\end{equation}

\begin{remark}
This is a little different from the generating set given in \cite{DayPutmanBirman}, which includes
generators of the form $\Mul{x^{-1}}{y}$ for $x \in X$ and $y \in Y$; however, these are
unnecessary here since $\Mul{x^{-1}}{y} = \Mul{x}{y} \Con{x}{y}^{-1}$.  The generators
$\Con{x}{y}$ were not included in the generating set in \cite{DayPutmanBirman}.  We give
the above form because it is a little more convenient for our purposes.
\end{remark}

The analogue of this for $\BKerIA{n}{k}$ is as follows.

\begin{maintheorem}
\label{maintheorem:generators}
The group $\BKerIA{n}{k}$ is generated by the finite set
\[\Set{$\Mul{x}{[y,z]}$}{$x \in X$, $y \in Y$, $z \in (X \cup Y) \setminus \{x,y\}$}
\cup \Set{$\Con{y}{z}$, $\Con{z}{y}$}{$y \in Y$, $z \in (X \cup Y) \setminus \{y\}$}.\]
\end{maintheorem}

\paragraph{The Torelli kernel is not finitely presentable.}
Though $\BKer{n}{k}$ is finitely generated, the authors proved in \cite{DayPutmanBirman} that it is not
finitely presentable if $n \geq 2$ and $k \geq 1$; in fact, $\HH_2(\BKer{n}{k};\Q)$ is infinite dimensional.  
If $\BKerIA{n}{k}$ were finitely presentable, then one could use the exact sequence \eqref{eqn:bkeriabker}
to build a finite presentation for $\BKer{n}{k}$.  We deduce that $\BKerIA{n}{k}$ is not finitely presentable.
Our second main theorem strengthens this observation.

\begin{maintheorem}
\label{maintheorem:notfp}
If $n \geq 2$ and $k \geq 1$, then $\HH_2(\BKerIA{n}{k};\Q)$ is infinite dimensional.  Consequently,
$\BKerIA{n}{k}$ is not finitely presentable.
\end{maintheorem}

\paragraph{Abelianization.}
The authors proved in \cite{DayPutmanBirman} that
\[\HH_1(\BKer{n}{k}) = \begin{cases}
\Z^{k(k-1)} & \text{if $n=0$},\\
\Z^{2kn} & \text{if $n>0$}.\end{cases}\]
These abelian quotients of $\BKer{n}{k}$ come from two sources.
\begin{compactitem}
\item The restriction of the map $\Aut(F_{n,k}) \rightarrow \GL_{n+k}(\Z)$ to $\BKer{n}{k}$, which
has image $\Z^{kn}$ (see the exact sequence \eqref{eqn:bkeriabker}).
\item The Johnson homomorphisms, which are homomorphisms 
\[\tau\co \IA_{n,k} \rightarrow \Hom(\Z^{n+k},\textstyle{\bigwedge^2} \Z^{n+k})\]
constructed from the action of $\IA_{n,k}$ on the second nilpotent truncation of $F_{n,k}$ (see
\S \ref{section:abelianization} below).  If $n=0$, then $\BKer{n}{k} \subset \IA_{n,k}$ and the restriction
of $\tau$ to $\BKer{n}{k}$ has image $\Z^{k(k-1)}$; this provides the entire abelianization.  If $n>0$, then
$\BKer{n}{k}$ does not lie in $\IA_{n,k}$ and we cannot use the Johnson homomorphism directly; however, in
\cite{DayPutmanBirman} we construct a modified version of it which is defined on $\BKer{n}{k}$ and has image
$\Z^{kn}$.
\end{compactitem}
The analogue of these calculations for $\BKerIA{n}{k}$ is as follows.

\begin{maintheorem}
\label{maintheorem:abelianization}
The group $\HH_1(\BKerIA{n}{k};\Z)$ is free abelian of rank
\[n(n-1)k +n \binom{k}{2} +2nk +k(k-1).\]
The abelianization map is given by the restriction
of the Johnson homomorphism to $\BKerIA{n}{k}$.
\end{maintheorem}

\begin{remark}
Theorem \ref{maintheorem:abelianization} is related to the fact that the Johnson homomorphism gives
the abelianization of $\IA_{n,k}$, a theorem which was proved independently by
Farb \cite{FarbAbel}, Cohen--Pakianathan \cite{CohenPakianathan}, and Kawazumi \cite{Kawazumi}.
\end{remark}

\paragraph{Finite L-presentation.}
Our final theorem gives an infinite presentation for the group $\BKerIA{n}{1}$ by generators and relations.  
Though this result may appear technical,
it is actually the most important theorem in this paper for our study in \cite{DayPutmanH2} of $\HH_2(\IA_n;\Z)$.
To simplify our notation, we will
write the generators of $F_{n,1}$ as $\{x_1,\ldots,x_n,y\}$.  The generators for our presentation
will be the finite set
\[S_K\coloneq\Set{$\Mul{x^{\alpha}}{[y^{\beta},z^{\gamma}]}$}{$x \in X$, $z \in X \setminus \{x\}, \alpha,\beta,\gamma = \pm 1$}
\cup \Set{$\Con{y}{x}$, $\Con{x}{y}$}{$x \in X$}.\]
This is larger than the generating set given by Theorem \ref{maintheorem:generators}; using
$S_K$ will simplify our relations.
By Theorem \ref{maintheorem:notfp}, the set of relations will have to be infinite.  They will be
generated from a finite list of relations by a simple recursive procedure which we will encode
using the notion of an L-presentation, which was introduced by Bartholdi \cite{BartholdiBranch}.

A {\em finite L-presentation} for a group $G$ is a triple $\LPres{S}{R^0}{E}$ as follows.
\begin{compactitem}
\item $S$ is a finite generating set for $G$.
\item $R^0$ is a finite subset of the free group $F(S)$ on $S$ 
consisting of relations for $G$.  It is not necessarily a complete set of relations.
\item $E$ is a finite subset of $\End(F(S))$.
\end{compactitem}
This triple must satisfy the following.  Let $M \subset \End(F(S))$ be the monoid
generated by $E$.  Define $R = \Set{$f(r)$}{$f \in M$, $r \in R^0$}$.  Then we require
that $G = \Pres{S}{R}$.
Each element of $E$ descends to an element of $\End(G)$; we call the resulting subset
$\widetilde{E} \subset \End(G)$ the {\em induced endomorphisms} of our L-presentation.

In this paper, the induced endomorphisms of our L-presentations will actually be automorphisms.
Thus in the context of this paper one should think of an L-presentation as a group presentation
incorporating certain symmetries of a group.  Here is an example.

\begin{example}
Fix $\ell \geq 1$.
Let $S = \Set{$z_i$}{$i \in \Z/\ell$}$ and $R^0 = \{z_0^2\}$.  Let $\psi \colon F(S) \rightarrow F(S)$ be the
homomorphism defined via the formula $\psi(z_i) = z_{i+1}$.  Then $\LPres{S}{R^0}{\{\psi\}}$ is a finite $L$-presentation
for the free product of $\ell$ copies of $\Z/2$.
\end{example}

We now return to the automorphism group of a free group.  The group $\BKer{n}{1}$ is a normal subgroup
of $\AutFB{n}{1}$, so $\AutFB{n}{1}$ acts on $\BKer{n}{1}$ by conjugation.  In \cite{DayPutmanBirman}, the
authors constructed a finite L-presentation for $\BKer{n}{1}$ whose set of induced endomorphisms 
generates 
\[\AutFB{n}{1} \subset \Aut(\BKer{n}{1}) \subset \End(\BKer{n}{1}).\]
The group $\BKerIA{n}{1}$ is also a normal subgroup of $\AutFB{n}{1}$, and hence
$\AutFB{n}{1}$ acts on $\BKerIA{n}{1}$ by conjugation.  Our final main theorem is as follows.

\begin{maintheorem}
\label{maintheorem:lpresentation}
For all $n \geq 2$, there exists a finite L-presentation
$\BKerIA{n}{1} = \LPres{S_{K}}{R_{K}^0}{E_{K}}$
whose set of induced endomorphisms generates $\AutFB{n}{1} \subset \Aut(\BKerIA{n}{1}) \subset \End(\BKerIA{n}{1})$.
\end{maintheorem}

See the tables in \S \ref{section:lpresentation} for explicit lists enumerating $R_{K}^0$ and $E_K$.

\paragraph{Verifying the L-presentation.}
We obtained the list of relations in $R_K^0$ by starting with a guess of a presentation
and then trying to run the following proof sketch.  Every time it failed, that failure revealed
a relation we had missed.  Let $\Gamma_n$ be the group given by the purported presentation
in Theorem \ref{maintheorem:lpresentation}.  There is a natural surjection $\Gamma_n \rightarrow \BKerIA{n}{1}$ that
we want to prove is an isomorphism.  As we will see in \S \ref{section:lpresentation} below, we have
a short exact sequence
\begin{equation}
\label{eqn:bkeriaseq}
1 \longrightarrow \BKerIA{n}{1} \longrightarrow \AutFB{n}{1} \stackrel{\rho}{\longrightarrow} \zasd \longrightarrow 1,
\end{equation}
where $\Aut(F_n)$ acts on $\Z^n$ via the natural surjection $\Aut(F_n) \rightarrow \GL_n(\Z)$.  The heart
of our proof is the construction of a similar extension $\Delta_n$ of $\zasd$ by $\Gamma_n$ which
fits into a commutative diagram
\begin{equation}
\label{eqn:maindiagram}
\begin{CD}
1  @>>> \Gamma_n        @>>> \Delta_n       @>>> \zasd @>>> 1  \\
@.      @VVV               @VVV              @VV{=}V                     @. \\
1  @>>> \BKerIA{n}{1} @>>> \AutFB{n}{1} @>{\rho}>> \zasd @>>> 1.
\end{CD}
\end{equation}
This construction is very involved; the exact sequences in \eqref{eqn:maindiagram} do not split, and constructing
group extensions with nonabelian kernels is delicate.  We will say more
about how we do this in the next paragraph.
In any case,
once we have constructed \eqref{eqn:maindiagram} we can use a known presentation of
$\AutFB{n}{1}$ due to Jensen--Wahl \cite{JensenWahl} to show that the map $\Delta_n \rightarrow \AutFB{n}{1}$
is an isomorphism.  The five-lemma then implies that the map $\Gamma_n \rightarrow \BKerIA{n}{1}$ is an isomorphism,
as desired.

\paragraph{The trouble with non-split extensions.}
If
\begin{equation}
\label{eqn:genericex}
1 \longrightarrow K \longrightarrow G \longrightarrow Q \longrightarrow 1
\end{equation}
is a group extension and presentations of $Q$ and $K$ are known, then it is
straightforward to construct a presentation of $G$.  However, when constructing
the group $\Delta_n$ in \eqref{eqn:maindiagram} we have to confront a serious problem, namely
we need to first verify that the desired extension exists.  To put it another way,
it is clear how to combine a known presentation of $\zasd$ with our purported
presentation for $\Gamma_n$ to form a group $\Delta_n$ which fits into a commutative
diagram
\[\begin{CD}
   @.   \Gamma_n        @>>> \Delta_n       @>>> \zasd @>>> 1  \\
@.      @VVV               @VVV              @VV{=}V                     @. \\
1  @>>> \BKerIA{n}{1} @>>> \AutFB{n}{1} @>{\rho}>> \zasd @>>> 1.
\end{CD}\]
However, it is difficult to show that the map $\Gamma_n \rightarrow \Delta_n$ is injective. 
Standard techniques show that proving the existence of the extension \eqref{eqn:genericex}
is equivalent to constructing a sort of ``nonabelian $K$-valued $2$-cocycle'' on $Q$; see
\cite[\S IV.6]{BrownCohomology}.  Such a $2$-cocyle is {\em not} determined by its values
on generators for $Q$.  This holds even in the simple case of a central extension; the general
case is even worse.  It is therefore very difficult to construct such a $2$-cocycle using
generators and relations.

But the extensions in \eqref{eqn:maindiagram} we are trying to understand are very special.  
While they do not split, there do exist ``partial splittings'', namely homomorphisms
$\iota_1 \co \Z^n \rightarrow \AutFB{n}{1}$ and $\iota_2\co \Aut(F_n) \rightarrow \AutFB{n}{1}$
such that $\rho \circ \iota_1 = \text{id}$ and $\rho \circ \iota_2 = \text{id}$.  Letting
$\Lambda_1 = \rho^{-1}(\Z^n)$ and $\Lambda_2 = \rho^{-1}(\Aut(F_n))$ we therefore have
$\Lambda_1 \cong \BKerIA{n}{1} \rtimes \Z^n$ and 
$\Lambda_2 \cong \BKerIA{n}{1} \rtimes \Aut(F_n)$.  
The data needed to combine $\Lambda_1$ and $\Lambda_2$ into a group $\AutFB{n}{1}$ that fits into \eqref{eqn:bkeriaseq}
is what we will call a ``twisted bilinear map'' from 
$\Aut(F_n) \times \Z^n$ to $\BKerIA{n}{1}$.
The definition is complicated, so to give the flavor of it in this introduction we will discuss a simpler
situation.

\paragraph{Splicing together direct products.}
Let $A$ and $B$ and $K$ be abelian groups.  We want to construct a not necessarily abelian group $G$
with the following property.
\begin{compactitem}
\item There is a short exact sequence
\[1 \longrightarrow K \longrightarrow G \stackrel{\rho}{\longrightarrow} A \times B \longrightarrow 1\]
together with homomorphisms $\iota_1\co A \rightarrow G$ and $\iota_2\co B \rightarrow G$ such that
$\rho^{-1}(A)$ and $\rho^{-1}(B)$ are the internal direct products $K \times \iota_1(A)$ and
$K \times \iota_2(B)$, respectively.
\end{compactitem}
Given this data, we can define a set map $\lambda\co A \times B \rightarrow K$ via the
formula
\[\lambda(a,b) = [\iota_1(a),\iota_2(b)] \quad \quad (a \in A, b \in B);\]
here the bracket is the commutator bracket in $G$.
It is easy to see that $\lambda$ is bilinear.  Conversely, given a bilinear map $\lambda\co A \times B \rightarrow K$
we can construct a group $G$ with the above properties by letting $G$ consist of
all triples $(k,b,a) \in K \times B \times A$ with the multiplication
\[(k,b,a)(k',b',a') = (k+k'+\phi(a,b'),b+b',a+a').\]
The bilinearity of $\phi$ is needed for this multiplication to be associative.

\paragraph{Adding the twisting.}
The groups we are interested in fit into semidirect products, so we will have
to incorporate the various group actions into our bilinear maps.  The key
property of the resulting theory of twisted bilinear maps is that (unlike
general $2$-cocycles but like ordinary bilinear maps) they are determined by their values on generators.  Letting
$\Gamma_n$ be the group in \eqref{eqn:maindiagram}, we will 
therefore be able to use combinatorial group theory to construct
an appropriate twisted bilinear map $Z^n \times \Aut(F_n) \rightarrow \Gamma_n$ 
that behaves like the twisted bilinear map $Z^n \times \Aut(F_n) \rightarrow \BKerIA{n}{1}$
that determines $\AutFB{n}{1}$.  This will allow us to construct the group 
$\Delta_n$ fitting into \eqref{eqn:maindiagram} and complete the proof
of Theorem \ref{maintheorem:lpresentation}.

\paragraph{Outline.}
We prove Theorem \ref{maintheorem:generators} in \S \ref{section:generators}, Theorem \ref{maintheorem:notfp} in
\S \ref{section:notfp}, and Theorem \ref{maintheorem:abelianization} in \S \ref{section:abelianization}.
Preliminaries for the proof Theorem \ref{maintheorem:lpresentation} are in 
\S \ref{section:lpresentationprelim},
and the proof itself appears in \S \ref{section:lpresentation}.  The proof
of Theorem \ref{maintheorem:lpresentation} depends on computer calculations
that are described in \S \ref{section:computer}.

\section{Generators}
\label{section:generators}

In this section, we prove Theorem \ref{maintheorem:generators}.  Letting
$X = \{x_1,\ldots,x_n\}$ and $Y = \{y_1,\ldots,y_k\}$, recall that this theorem asserts
that the set
\[T\coloneq\Set{$\Mul{x}{[y,z]}$}{$x \in X$, $y \in Y$, $z \in (X \cup Y) \setminus \{x,y\}$}
\cup \Set{$\Con{y}{z}$, $\Con{z}{y}$}{$y \in Y$, $z \in (X \cup Y) \setminus \{y\}$}\]
generates $\BKerIA{n}{k}$.  
\begin{proof}[Proof of Theorem~\ref{maintheorem:generators}]
The key is the exact sequence
\[1 \longrightarrow \BKerIA{n}{k} \longrightarrow \BKer{n}{k} \stackrel{\rho}{\longrightarrow} \Z^{nk} \longrightarrow 1\]
discussed in the introduction (see \eqref{eqn:bkeriabker}).  Define
\[S_1 = \Set{$\Mul{x}{y}$}{$x \in X$, $y \in Y$} \quad \text{and} \quad S_2 = \Set{$\Con{y}{z}$, $\Con{z}{y}$}{$y \in Y$, $z \in (X \cup Y) \setminus \{y\}$}.\]
As we discussed in the introduction (see the equation \eqref{eqn:bkergen} and the remark following it), 
the authors proved in \cite{DayPutmanBirman} that $S_1 \cup S_2$ generates $\BKer{n}{k}$.  
We have $S_2 \subset \BKerIA{n}{k} = \ker(\rho)$.  Also, $\rho$ maps the elements of $S_1$
to a basis of $\Z^{nk}$.  We therefore see that $\Z^{nk}$ is the quotient
of $\BKer{n}{k}$ by the normal closure of the set $S_1' \cup S_2$, where
\begin{align*}
S_1' &= \Set{$[s,s']$}{$s,s' \in S_1$} \\
&= \Set{$[\Mul{x}{y},\Mul{x}{y'}]$}{$x \in X$, $y,y' \in Y, y \neq y'$} \\
&\quad\quad\quad\cup \Set{$[\Mul{x}{y},\Mul{x'}{y'}]$}{$x,x' \in X$, $y,y' \in Y$, $x \neq x'$} \\
&= \Set{$\Mul{x}{[y,y']}$}{$x \in X$, $y,y' \in Y$, $y \neq y'$}.
\end{align*}
Here we are using the fact that $[\Mul{x}{y},\Mul{x'}{y}] = 1$ for $x,x' \in X$ and $y \in Y$ with $x \neq x'$.  Since
$S_1' \cup S_2 \subset T$, we conclude that $T$ normally generates $\BKerIA{n}{k}$.

Letting $G \subset \BKerIA{n}{k}$ be the subgroup generated by $T$, it is therefore enough to prove that
$G$ is a normal subgroup.  To do this, it is enough to prove that for $s \in S_1 \cup S_2$ and $t \in T$, we
have $s t s^{-1} \in G$ and $s^{-1} t s \in G$.  In fact, since $S_2 \subset T$ it is enough to do this for
$s \in S_1$.  
The identities that show this are in Table
\ref{table:gisnormal}.
\end{proof}

\begin{table}
\tiny
\centering
\begin{tabular}{c|c|c}
$t \in T$ & $sts^{-1}$ & $s^{-1}ts$ \\ 
\hline
$\Mul{x_a}{[y_d,x_b]}$ & $\Con{x_a}{y_d} \Mul{x_a}{[y_d,x_b]} \Con{x_a}{y_d}^{-1}$ & $\Con{x_a}{y_d}^{-1} \Mul{x_a}{[y_d,x_b]} \Con{x_a}{y_d}$\\
$\Mul{x_b}{[y_d,x_a]}$ & $\Con{x_a}{y_d} \Mul{x_b}{[y_d,x_a]} \Con{x_a}{y_d}^{-1}$ & $\Con{x_a}{y_d}^{-1} \Mul{x_b}{[y_d,x_a]} \Con{x_a}{y_d}$\\
$\Mul{x_a}{[y_e,x_b]}$ & $\Con{x_a}{y_d} \Mul{x_a}{[y_e,x_b]} \Con{x_a}{y_d}^{-1}$ & $\Con{x_a}{y_d}^{-1} \Mul{x_a}{[y_e,x_b]} \Con{x_a}{y_d}$\\
$\Mul{x_b}{[y_e,x_a]}$ & $\Con{x_b}{y_d}^{-1} \Mul{x_b}{[y_e,x_a]} \Con{x_b}{y_d} \Mul{x_b}{[y_e,y_d]}$ & $\Con{x_b}{y_d} \Mul{x_b}{[y_e,x_a]} \Mul{x_b}{[y_d,y_e]}\Con{x_b}{y_d}^{-1} $\\
$\Mul{x_a}{[y_d,y_e]}$ & $\Con{x_a}{y_d} \Mul{x_a}{[y_d,y_e]} \Con{x_a}{y_d}^{-1}$ & $\Con{x_a}{y_d}^{-1} \Mul{x_a}{[y_d,y_e]} \Con{x_a}{y_d}$ \\
$\Mul{x_a}{[y_e,y_f]}$ & $\Con{x_a}{y_d}\Mul{x_a}{[y_e,y_f]}\Con{x_a}{y_d}^{-1}$ & $\Con{x_a}{y_d}^{-1}\Mul{x_a}{[y_e,y_f]}\Con{x_a}{y_d}$ \\

$\Con{y_d}{y_e}$ & $\Con{x_a}{y_d}\Mul{x_a}{[y_d,y_e]}\Con{x_a}{y_d}^{-1}\Con{y_d}{y_e}$ & $\Mul{x_a}{[y_e,y_d]}\Con{y_d}{y_e}$ \\
$\Con{y_d}{x_a}$ & $\Con{x_a}{y_d}\Con{y_d}{x_a}$ & $\Con{x_a}{y_d}^{-1}\Con{y_d}{x_a}$\\
$\Con{y_d}{x_b}$ & $\Con{y_d}{x_b}\Con{x_a}{y_d}\Con{y_d}{x_b}^{-1}\Mul{x_a}{[y_d,x_b]}\Con{y_d}{x_b}\Con{x_a}{y_d}$ & $\Con{x_a}{y_d}\Con{y_d}{x_b}^{-1}\Mul{x_a}{[y_d,x_b]}^{-1}\Con{y_d}{x_b}$ \\
$\Con{y_e}{x_a}$ & $\Con{y_e}{x_a}\Con{y_e}{y_d}$ & $\Con{y_e}{x_a}\Con{y_e}{y_d}^{-1}$\\
$\Con{x_a}{y_e}$ & $\Con{x_a}{y_e}\Con{x_a}{y_d}\Mul{x_a}{[y_e,y_d]}\Con{x_a}{y_d}^{-1}$ & $\Con{x_a}{y_e}\Mul{x_a}{[y_d,y_e]}$ \\
\end{tabular}
\caption{Fix $s = \Mul{x_a}{y_d}$.  This table shows how to write $sts^{-1}$ and $s^{-1}ts$ as a word
in $T$ for all $t \in T$.  Basis elements with distinct subscripts are assumed to be distinct.  If a formula is not listed, then
$sts^{-1} = s^{-1}ts = t$.  All these formulas can be easily proved by checking the effect
of the indicated automorphisms on a basis for the free group.}
\label{table:gisnormal}
\end{table}

\section{The Torelli kernel is not finitely presentable}
\label{section:notfp}

In this section, we prove Theorem \ref{maintheorem:notfp}.  Recall that this theorem asserts that $\HH_2(\BKerIA{n}{k};\Q)$ is
infinite dimensional when $n \geq 2$ and $k \geq 1$.  
\begin{proof}[Proof of Theorem~\ref{maintheorem:notfp}]
Consider the Hochschild--Serre spectral sequence associated
to the short exact sequence
\[1 \longrightarrow \BKerIA{n}{k} \longrightarrow \BKer{n}{k} \stackrel{\rho}{\longrightarrow} \Z^{nk} \longrightarrow 1\]
discussed in the introduction (see \eqref{eqn:bkeriabker}).   It is of the form
\[E^2_{pq} = \HH_p(\Z^{nk};\HH_q(\BKerIA{n}{k};\Q)) \Rightarrow \HH_{p+q}(\BKer{n}{k};\Q).\]
The authors proved in \cite{DayPutmanBirman} that $\HH_2(\BKer{n}{k};\Q)$ is infinite dimensional, so at least one
of $E^2_{20}$ and $E^2_{11}$ and $E^2_{02}$ must be infinite dimensional.  Clearly
$E^2_{20} = \HH_2(\Z^{nk};\Q)$ is finite dimensional.  Also, Theorem \ref{maintheorem:generators} implies that
$\HH_1(\BKerIA{n}{k};\Q)$ is finite-dimensional, so $E^2_{11} = \HH_1(\Z^{nk};\HH_1(\BKerIA{n}{k};\Q))$ is finite dimensional
We conclude that $E^2_{02} = \HH_0(\Z^{nk};\HH_2(\BKerIA{n}{k};\Q))$ is infinite dimensional, so
$\HH_2(\BKerIA{n}{k};\Q)$ is infinite dimensional, as desired.
\end{proof}

\section{Abelianization}
\label{section:abelianization}

In this section, we prove Theorem \ref{maintheorem:abelianization}.  Recall that this theorem asserts that
$\HH_1(\BKerIA{n}{k};\Z)$ is free abelian and that the abelianization map is given by the restriction
of the Johnson homomorphism to $\BKerIA{n}{k}$.  
The theorem also gives the rank of the abelianization of $\BKerIA{n}{k}$ as a polynomial in $n$ and $k$.
\begin{proof}[Proof of Theorem~\ref{maintheorem:abelianization}]
We begin by recalling the definition of the Johnson
homomorphism; see \cite{SatohJohnson} for more details and references.  Let $\pi\co [F_{n,k},F_{n,k}] \rightarrow \bigwedge^2 \Z^{n+k}$
be the projection whose kernel is $[F_{n,k},[F_{n,k},F_{n,k}]]$.  This map satisfies $\pi([z,z']) = [z] \wedge [z']$ for
$z,z' \in F_{n,k}$; here $[z],[z'] \in \Z^{n+k}$ are the images of $z$ and $z'$ in the abelianization of $F_{n,k}$.  The
Johnson homomorphism is then a homomorphism
\[\tau\co \IA_{n,k} \rightarrow \Hom(\Z^{n+k},\textstyle{\bigwedge^2} \Z^{n+k})\]
that satisfies the formula
\[\tau(f)([z]) = \pi(f(z) z^{-1}) \quad \quad (f \in \IA_{n,k}, z \in F_{n,k}).\]
Letting $X = \{x_1,\ldots,x_n\}$ and $Y = \{y_1,\ldots,y_k\}$, the Johnson homomorphism has the following
effect on the basic elements of $\IA_{n,k}$ defined in the introduction.
\begin{compactitem}
\item For distinct $z,z' \in X \cup Y$, we have
\[\tau(\Con{z}{z'})([w]) = \begin{cases}
[z] \wedge [z'] & \text{if $w = z$},\\
0 & \text{otherwise} \end{cases}
\quad \quad (w \in X \cup Y).\]
\item For distinct $z,z',z'' \in X \cup Y$, we have
\[\tau(\Mul{z}{[z',z'']})([w]) = \begin{cases}
[z'] \wedge [z''] & \text{if $w = z$},\\
0 & \text{otherwise} \end{cases}
\quad \quad (w \in X \cup Y).\]
\end{compactitem}
Set
\begin{align*}
T = &\Set{$\Mul{x}{[y,x']}$}{$x \in X$, $y \in Y$, $x' \in X \setminus \{x\}$} \\
&\cup\Set{$\Mul{x}{[y_a,y_b]}$}{$x \in X$, $1 \leq a < b \leq k$} \\
&\cup\Set{$\Con{y}{z}$, $\Con{z}{y}$}{$y \in Y$, $z \in (X \cup Y) \setminus \{y\}$}.
\end{align*}
Since $\Mul{x}{[y_b,y_a]} = \Mul{x}{[y_a,y_b]}^{-1}$ for
$x \in X$ and $1 \leq a < b \leq k$, Theorem \ref{maintheorem:generators} implies that $T$ generates
$\BKerIA{n}{k}$.  Examining the above formulas, 
we see that $\tau$ takes $T$ injectively to a linearly independent subset of the free abelian group
$\Hom(\Z^{n+k},\bigwedge^2 \Z^{n+k})$.  This implies that if $u$ is an element of the free group $F(T)$ on $T$
which maps to a relation in $\BKerIA{n}{k}$, then $u \in [F(T),F(T)]$ (otherwise, $\tau$ would take the
image of $u$ in $\BKerIA{n}{k}$ to a nontrivial element of $\Hom(\Z^{n+k},\bigwedge^2 \Z^{n+k})$).
We conclude that $\tau$ induces the abelianization of $\BKerIA{n}{k}$ and that
$\HH_1(\BKerIA{n}{k};\Z) = \Z^{|T|}$.
This is free abelian of rank
\[|T|= n(n-1)k +n \binom{k}{2}+2nk +k(k-1).\qedhere\]
\end{proof}

\section{Preliminaries for the proof of Theorem \ref{maintheorem:lpresentation}}
\label{section:lpresentationprelim}

The rest of this paper is devoted to proving Theorem \ref{maintheorem:lpresentation}, 
which gives a finite L-presentation for $\BKerIA{n}{1}$.  This section contains three
subsections of preliminaries: \S \ref{section:relatingkernels} constructs a needed exact
sequence, \S \ref{section:twistedbilinear} discusses twisted bilinear maps, and
\S \ref{section:presentations} discusses presentations for some related groups.

To simplify our notation, 
we will set $y = y_1$, so $\{x_1,\ldots,x_n,y\}$ is the basis for $F_{n,1}$.  When
writing matrices in $\GL_{n+1}(\Z)$, we will always use the basis $\{[x_1],\ldots,[x_n],[y]\}$
for $\Z^{n+1}$.  

\subsection{Relating the two kernels}
\label{section:relatingkernels}

This is the first of three preliminary sections for the proof of Theorem \ref{maintheorem:lpresentation}.
In it, we construct the exact sequence 
\[1 \longrightarrow \BKerIA{n}{1} \longrightarrow \AutFB{n}{1} \stackrel{\rho}{\longrightarrow} \zasd \longrightarrow 1\]
discussed in the introduction (see Lemma \ref{lemma:relatingkernels} below).  
We first address an irritating technical point.  Throughout this
paper, all group actions are left actions.  In particular, elements of $\Z^n$ will be regarded
as column vectors and matrices in $\GL_n(\Z)$ act on these column vectors on the left (we have
already silently used this convention when we wrote matrices).  However, it turns out that the action
of $\Aut(F_n)$ on $\Z^n$ in the semidirect product appearing the above exact sequence is induced
by the natural {\em right} action of $\GL_n(\Z)$ on row vectors.  We do not wish to mix up right and
left actions, so we convert this into a left action and define $\Z^n \rtimes_r \GL_n(\Z)$ to be the semidirect
product associated to the action of $\GL_n(\Z)$ on $\Z^n$ defined by the formula
\[M \cdot z = (M^{-1})^{t} z \quad \quad (M \in \GL_n(\Z), z \in \Z^n).\]
To understand this formula, observe that $(M^{-1})^{t} z$ is the transpose of $(z^t) M^{-1}$; the inverse
appears because we are converting a right action into a left action.  We then have the following.

\begin{lemma}
\label{lemma:identifystabilizer}
The stabilizer subgroup $(\GL_{n+1}(\Z))_{[y]}$ is isomorphic to $\Z^n \rtimes_r \GL_n(\Z)$.
\end{lemma}
\begin{proof}
We define a homomorphism $\psi\co (\GL_{n+1}(\Z))_{[y]} \rightarrow \Z^n \rtimes_r \GL_n(\Z)$ 
as follows.
Consider $M \in (\GL_{n+1}(\Z))_{[y]}$.  There exist $\widehat{M} \in \GL_n(\Z)$ and
$\overline{M} \in \Z^n$ such that
\[M = \left(\begin{array}{c|c} \widehat{M} & 0 \\ \hline \overline{M}^t & 1 \end{array}\right);\]
here we are using our convention that elements of $\Z^n$ are column vectors, so the transpose $\overline{M}^t$ of
$\overline{M} \in \Z^n$ is a row.
We then define $\psi(M) = \left(\left(\widehat{M}^{-1}\right)^t \overline{M}, \widehat{M}\right)$.  To see that this is a homomorphism, observe that
for $M_1,M_2 \in (\GL_{n+1}(\Z))_{[y]}$ we have
\[M_1 M_2 = \left(\begin{array}{c|c} \widehat{M}_1 & 0 \\ \hline \overline{M}_1^t & 1 \end{array}\right) 
\left(\begin{array}{c|c} \widehat{M}_2 & 0 \\ \hline \overline{M}_2^t & 1 \end{array}\right) = 
\left(\begin{array}{c|c} \widehat{M}_1 \widehat{M}_2 & 0 \\ \hline \overline{M}_1^t \widehat{M}_2 + \overline{M}_2^t & 1 \end{array}\right)
=\left(\begin{array}{c|c} \widehat{M}_1 \widehat{M}_2 & 0 \\ \hline (\widehat{M}_2^t \overline{M}_1 + \overline{M}_2)^t & 1 \end{array}\right),\]
and hence
\begin{align*}
\psi(M_1) \psi(M_2) &= \left(\left(\widehat{M}_1^{-1}\right)^t \overline{M}_1, \widehat{M}_1\right) \left(\left(\widehat{M}_2^{-1}\right)^t \overline{M}_2, \widehat{M}_2\right) \\
&= \left(\left(\widehat{M}_1^{-1}\right)^t \overline{M}_1 + \left(\widehat{M}_1^{-1}\right)^t \left(\widehat{M}_2^{-1}\right)^t \overline{M}_2, \widehat{M}_1 \widehat{M}_2\right) \\
&= \left(\left(\widehat{M}_1^{-1}\right)^t \left(\widehat{M}_2^{-1}\right)^t \left(\widehat{M}_2^t \overline{M}_1 + \overline{M}_2\right), \widehat{M}_1 \widehat{M}_2\right) \\
&= \psi(M_1 M_2).
\end{align*}
That $\psi$ is a bijection is obvious.
\end{proof}

\begin{remark}
There is an isomorphism between $\Z^n \rtimes_r \GL_n(\Z)$ 
and the semidirect product of $\Z^n$ and $\GL_n(\Z)$ with respect to the
standard left action of $\GL_n(\Z)$ on $\Z^n$.  However, this isomorphism acts
as the inverse transpose on the $\GL_n(\Z)$ factor, and to keep our formulas from
getting out of hand we want to not change this factor.  Throughout this paper,
we will use the explicit isomorphism described in the proof of
Lemma \ref{lemma:identifystabilizer}.
\end{remark}

Define $\zasdr$ to be the semidirect product induced by the action of
$\Aut(F_n)$ on $\Z^n$ obtained by composing the projection $\Aut(F_n) \rightarrow \GL_n(\Z)$
with the action of $\GL_n(\Z)$ on $\Z^n$ discussed above.  We then have the following
lemma, which is the main result of this section.

\begin{lemma}
\label{lemma:relatingkernels}
There is a short exact sequence
\[1 \longrightarrow \BKerIA{n}{1} \longrightarrow \AutFB{n}{1} \stackrel{\rho}{\longrightarrow} \zasdr \longrightarrow 1.\]
Also, there exist homomorphisms $\iota_1\co\Aut(F_n) \rightarrow \AutFB{n}{1}$ and 
$\iota_2\co \Z^n \rightarrow \AutFB{n}{1}$ such that $\rho \circ \iota_1$ 
and $\rho \circ \iota_2$ are the standard inclusions of $\Aut(F_n)$ and $\Z^n$ into $\zasdr$ respectively.
\end{lemma}
\begin{proof}
Let $\pi_1\co\AutFB{n}{1} \rightarrow \Aut(F_n)$ be the map induced by the projection
$F_{n,1} \rightarrow F_n$ whose kernel is normally generated by $y$.  Also, let
$\pi_2\co\AutFB{n}{1} \rightarrow \Z^n \rtimes_r \GL_n(\Z)$ be the composition
\[\AutFB{n}{1} \longrightarrow (\GL_{n+1}(\Z))_{[y]} \stackrel{\cong}{\longrightarrow} \Z^n \rtimes_r \GL_n(\Z),\]
where the second map is the isomorphism given by Lemma \ref{lemma:identifystabilizer}.
By definition, $\BKer{n}{1} = \ker(\pi_1)$ and $\IAB{n}{1} = \ker(\pi_2)$.  Recalling
that $\BKerIA{n}{1} = \BKer{n}{1} \cap \IAB{n}{1}$, it follows that
$\BKerIA{n}{1} = \ker(\rho)$, where $\rho$ is the composition
\[\AutFB{n}{1} \stackrel{\pi_1 \oplus \pi_2}{\longrightarrow} \left(\Aut(F_n)\right) \oplus \left(\Z^n \rtimes_r \GL_n\left(\Z\right)\right).\]
Let $\eta\co\Aut(F_n) \rightarrow \GL_n(\Z)$ be the natural projection.  The image
of $\rho$ is contained in the subgroup
\[\Set{$\left(f,\left(z,\eta\left(f\right)\right)\right)$}{$f \in \Aut(F_n)$, $z \in \Z^n$}
\subset \left(\Aut(F_n)\right) \oplus \left(\Z^n \rtimes_r \GL_n\left(\Z\right)\right),\]
which is clearly isomorphic to $\zasdr$.  We can therefore regard $\rho$ as a homomorphism
$\rho\co\AutFB{n}{1} \rightarrow \zasdr$, and we have an exact sequence
\[1 \longrightarrow \BKerIA{n}{1} \longrightarrow \AutFB{n}{1} \stackrel{\rho}{\longrightarrow} \zasdr.\]
Let $\iota_1\co\Aut(F_n) \rightarrow \AutFB{n}{1}$ be the evident inclusion whose image
is the stabilizer subgroup $(\AutFB{n}{1})_y$ and let $\iota_2\co\Z^n \rightarrow \AutFB{n}{1}$
be the map defined via the formula
\[\iota_2(z_1,\ldots,z_n) = \Mul{x_1}{y}^{z_1} \Mul{x_2}{y}^{z_2} \cdots \Mul{x_n}{y}^{z_n} \quad \quad (z_1,\ldots,z_n \in \Z),\]
where the automorphisms $\Mul{x_i}{y}$ are as in the introduction.  The map
$\iota_2$ is a homomorphism because the $\Mul{x_i}{y}$ commute.  It is clear that
$\rho \circ \iota_1 = \text{id}$ and $\rho \circ \iota_2 = \text{id}$.  This implies
that $\rho$ is surjective, and the lemma follows.
\end{proof}

\subsection{Twisted bilinear maps and group extensions}
\label{section:twistedbilinear}

This is the second section containing preliminaries for the proof of Theorem \ref{maintheorem:lpresentation}.
In it, we discuss the theory of twisted bilinear maps alluded to in the introduction.
Throughout this section, let $A$ and $B$ and $K$ be groups equipped with the following left actions.
\begin{compactitem}
\item The group $A$ acts on $B$; for $a \in A$ and $b \in B$, we will write
$\AB{a}{b}$ for the image of $b$ under the action of $a$.
\item The groups $A$ and $B$ both act on $K$.  For $k \in K$ and $a \in A$ and $b \in B$, we
will write $\AK{a}{k}$ and $\BK{b}{k}$ for the images of $k$ under the actions of $a$ and $b$,
respectively.
\end{compactitem}
A {\em twisted bilinear map} from $A \times B$ to $K$ is a set map $\lambda\co A \times B \rightarrow K$
satisfying the following three properties.
\begin{compactitem}
\item[TB1.\ ] For all $a \in A$ and $b_1,b_2 \in B$, we have
$\lambda(a,b_1 b_2) = \lambda(a,b_1) \cdot \BK{\AB{a}{b_1}}{\lambda(a,b_2)}$.
\item[TB2.\ ] For all $a_1,a_2 \in A$ and $b \in B$, we have
$\lambda(a_1 a_2,b) = \AK{a_1}{\lambda(a_2,b)} \cdot \lambda(a_1,\AB{a_2}{b})$.
\item[TB3.\ ] For all $a \in A$ and $b \in B$ and $k \in K$, we have
$\lambda(a,b) \cdot \BK{\AB{a}{b}}{\AK{a}{k}} \cdot \lambda(a,b)^{-1} = \AK{a}{\BK{b}{k}}$.
\end{compactitem}
Observe that this reduces to the definition of a bilinear map if all the actions are trivial and all
the groups are abelian.  The key example is as follows.

\begin{example}
\label{example:groupextend}
Consider a short exact sequence of groups
\[1 \longrightarrow K \longrightarrow G \stackrel{\rho}{\longrightarrow} B \rtimes A \longrightarrow 1\]
together with homomorphisms $\iota_1\co A \rightarrow G$ and $\iota_2\co B \rightarrow G$ such that
$\rho\circ \iota_1$ and $\rho\circ \iota_2$ are the standard inclusions of $A$ and $B$ in $B\rtimes A$ respectively.
Observe that this implies that
$\rho^{-1}(A)$ and $\rho^{-1}(B)$ are the internal semidirect products $K \rtimes \iota_1(A)$ and
$K \rtimes \iota_2(B)$, respectively.
Let $\AB{a}{b}$ be the action $a\cdot b$ defining the semidirect product $B\rtimes A$, and define actions of $A$ and $B$ on $K$ by
\[\AK{a}{k}=\iota_1(a) \cdot k \cdot \iota_1(a)^{-1} \quad \text{and} \quad  \BK{b}{k}=\iota_2(b) \cdot k \cdot \iota_2(b)^{-1} \]
for all $a \in A$ and $b \in B$ and $k \in K$.  
Define a set map $\lambda\co A \times B \rightarrow K$ via the
formula
\[\lambda(a,b) = \iota_1(a) \cdot \iota_2(b) \cdot \iota_1(a)^{-1} \iota_2(\AB{a}{b})^{-1}.\]
Note that this is a kind of ``twisted commutator" map;
it reduces to the commutator bracket if the action of $A$ on $B$ is trivial.  
Given these definitions, an
easy algebraic juggle shows that $\lambda$ is a twisted bilinear map.  We will say that
$\lambda$ is the twisted bilinear map {\em associated} to $G$ and $\iota_1$ and $\iota_2$.  
\end{example}

\begin{remark}
We take a moment to explain the different aspects of the definition of a twisted bilinear map.
For a fixed $a\in A$, we can twist the action of $B$ on $K$ by the action of $A$ on $B$ by $a$, to get an action $b\cdot k = \BK{\AB{a}{b}}{k}$.
Property TB1 simply states that the function $\lambda(a,\cdot)\co B\to K$ is a crossed homomorphism with respect to this action twisted by $a$.

Property TB2 is similar, but involves two kinds of twisting.
The set of functions $B\to K$ is a group with the pointwise product.
The group $A$ acts on this group in two ways.
The first is by post-composition: for $a\in A$ and $f\co B\to K$, define $a\cdot f$ by $(a\cdot f)(b)=\AK{a}{f(b)}$.
The second is by pre-composition, and is a right action: for $a, f$ as above, define $f\cdot a$ by $(f\cdot a)(b)=f(\AB{a}{b})$.
So property TB2 says that $\lambda$ is like a crossed homomorphism from $A$ to the group of functions $B\to K$, but simultaneously twisted by both of these actions.

Property TB3 can be explained in the context of Example~\ref{example:groupextend}.
As usual in group extensions, there is a well defined \emph{outer} action of $B\rtimes A$ on $K$: to act on $k$ by $(b,a)$, lift $(b,a)$ to $G$ and conjugate $k$ by this lift.
The conjugate of $k$ depends on the choice of lift, but the resulting map $B\rtimes A\to \Out(K)$ is well defined.
Using our maps $\iota_1$ and $\iota_2$, we have two ways to build lifts.
By the definition of the product in $B\rtimes A$, we have
\[(1,a)(b,1)=(\AB{a}{b},a)=(\AB{a}{b},1)(1,a).\]
So we may view the outer action of $(\AB{a}{b},a)$ on $K$ as coming from conjugation either by $\iota_1(a)\iota_2(b)$ or by $\iota_2(\AB{a}{b})\iota_1(a)$.
These conjugations are given by $k\mapsto\AK{a}{\BK{b}{k}}$ and $k\mapsto \BK{\AB{a}{b}}{\AK{a}{k}}$ respectively.
Since they define the same outer automorphism, they differ by conjugation by some element; TB3 says that $\lambda(a,b)$ is such an element.
\end{remark}

The following theorem shows that every twisted bilinear map is associated to some group
extension.

\begin{theorem}
\label{theorem:twistedextend}
Let the groups and actions be as above and let $\lambda\co A \times B \rightarrow K$ be a twisted
bilinear map.  Then $\lambda$ is the twisted bilinear map associated to some group $G$, some
short exact sequence 
\[1 \longrightarrow K \longrightarrow G \stackrel{\rho}{\longrightarrow} B \rtimes A \longrightarrow 1,\]
and some homomorphisms $\iota_1\co A \rightarrow G$ and $\iota_2\co B \rightarrow G$.
\end{theorem}
\begin{proof}
Set $Q = B \rtimes A$.
We will construct $G$ using the theory of nonabelian group extensions sketched in
\cite[\S IV.6]{BrownCohomology} and proven in detail in \cite{EilenbergMacLane}.  This
machine needs two inputs.
\begin{compactitem}
\item The first is a set map $\phi\co Q \rightarrow \Aut(K)$ satisfying 
$\phi(1) = \text{id}$.  
For $q \in Q$, we define $\phi(q) \in \Aut(K)$ as follows.
We can uniquely write $q = ba$ with $b \in B$ and $a \in A$.  We then define
\[\phi(q)(k) = \BK{b}{\AK{a}{k}} \quad \quad (k \in K).\]
\item The second is a set map $\gamma\co Q \times Q \rightarrow K$ satisfying
$\gamma(1,q) = \gamma(q,1) = 1$ for all $q \in Q$, which we define as follows.
Consider $q_1,q_2 \in Q$.  We can uniquely write $q_1 = b_1 a_1$ and
$q_2 = b_2 a_2$ with $b_1,b_2 \in B$ and $a_1, a_2 \in A$.  We then define
\[\gamma(q_1,q_2) = \BK{b_1}{\lambda(a_1,b_2)}.\]
\end{compactitem}
We remark that these pieces of data are not homomorphisms.  They must satisfy two key identities which
we will verify below in Claims \ref{claim:extend1} and \ref{claim:extend2}.  
We postpone these verifications momentarily to explain the output of the machine.

Let $G$ be the set of pairs $(k,q)$ with $K \in K$
and $q \in Q$.  Define a multiplication in $G$ via the formula
\[(k_1,q_1)(k_2,q_2) = \big(k_1 \cdot \phi(q_1)(k_2) \cdot \gamma(q_1,q_2), q_1 q_2\big).\]
The machine says that this $G$ is a group (the purpose of the postponed identities is to show that the above
multiplication is associative).  It clearly lies in a short exact sequence
\[1 \longrightarrow K \longrightarrow G \stackrel{\rho}{\longrightarrow} B \rtimes A \longrightarrow 1,\]
and we can define the desired homomorphisms $\iota_1\co A \rightarrow G$ and $\iota_2\co B \rightarrow G$
via the formulas $\iota_1(a) = (1,(1,a))$ and $\iota_2(b) = (1,(b,1))$.  An easy calculation shows that
the conclusions of the theorem are satisfied.

It remains to verify the two needed identities, which are as follows.

\begin{claimsa}
\label{claim:extend1}
For $q_1, q_2 \in Q$ and $k \in K$, we have
\[\phi(q_1)(\phi(q_2)(k)) = \gamma(q_1, q_2) \cdot \phi(q_1 q_2)(k) \cdot \gamma(q_1,q_2)^{-1}.\]
\end{claimsa}

For $i=1,2$, write $q_i = b_i a_i$ with $b_i \in B$ and $a_i \in A$, so
$q_1 q_2 = (b_1 \AB{a_1}{b_2})(a_1 a_2)$.  We then have
\[\begin{split}
\gamma(q_1, q_2) \cdot \phi(q_1 q_2)(k) \cdot \gamma(q_1,q_2)^{-1}
&= \BK{b_1}{\lambda(a_1,b_2)} \cdot \BK{b_1 \AB{a_1}{b_2}}{\AK{a_1 a_2}{k}} \cdot \BK{b_1}{\lambda(a_1,b_2)}^{-1}\\
&= \BK{b_1}{\lambda(a_1,b_2) \cdot \BK{\AB{a_1}{b_2}}{\AK{a_1}{\AK{a_2}{k}}} \cdot \lambda(a_1,b_2)^{-1}}.
\end{split}\]
Property TB3 of a twisted bilinear map says that this equals
\[\BK{b_1}{\AK{a_1}{\BK{b_2}{\AK{a_2}{k}}}} = \phi(q_1)(\phi(q_2)(k)),\]
as claimed.

\begin{claimsa}
\label{claim:extend2}
For $q_1, q_2, q_3 \in Q$, we have
$\gamma(q_1,q_2) \cdot \gamma(q_1 q_2, q_3) = \phi(q_1)(\gamma(q_2,q_3)) \cdot \gamma(q_1,q_2 q_3)$.
\end{claimsa}

For $i=1,2,3$, write $q_i = b_i a_i$ with $b_i \in B$ and $a_i \in A$.
We begin by examining the left side of the purported equality.  Since
$q_1 q_2 = (b_1 \AB{a_1}{b_2})(a_1 a_2)$, it equals
\[\begin{split}
\gamma(q_1,q_2) \cdot \gamma(q_1 q_2, q_3) 
&= \BK{b_1}{\lambda(a_1,b_2)} \cdot \BK{b_1 \AB{a_1}{b_2}}{\lambda(a_1 a_2, b_3)} \\
&= \BK{b_1}{\lambda(a_1,b_2) \cdot \BK{\AB{a_1}{b_2}}{\lambda(a_1 a_2, b_3)}}.
\end{split}\]
Using property TB2 of a twisted bilinear map, this equals
\begin{equation}
\label{eqn:intermediate1}
\BK{b_1}{\lambda(a_1,b_2) \cdot \BK{\AB{a_1}{b_2}}{\AK{a_1}{\lambda(a_2,b_3)} \cdot \lambda(a_1,\AB{a_2}{b_3})}} 
\end{equation}
Property TB3 of a twisted bilinear map with $a = a_1$ and $b = b_2$ and 
$k = \lambda(a_2,b_3)$ says that
\[\lambda(a_1,b_2) \cdot \BK{\AB{a_1}{b_2}}{\AK{a_1}{\lambda(a_2,b_3)}} = \AK{a_1}{\BK{b_2}{\lambda(a_2,b_3)}} \cdot \lambda(a_1,b_2).\]
Plugging this into \eqref{eqn:intermediate1} (and remembering to distribute the 
$\beta_{\AB{a_1}{b_2}}$ over the last term), we get
\begin{equation}
\label{eqn:intermediate2}
\BK{b_1}{\AK{a_1}{\BK{b_2}{\lambda(a_2,b_3)}} \cdot \lambda(a_1,b_2) \cdot \BK{\AB{a_1}{b_2}}{\lambda(a_1,\AB{a_2}{b_3})}}.
\end{equation}
We now turn to the right hand side of the purported equality.  Since
$q_2 q_3 = (b_2 \AB{a_2}{b_3})(a_2 a_3)$, it equals
\[\BK{b_1}{\AK{a_1}{\BK{b_2}{\lambda(a_2,b_3)}}} \cdot \BK{b_1}{\lambda(a_1,b_2 \AB{a_2}{b_3})} = \BK{b_1}{\AK{a_1}{\BK{b_2}{\lambda(a_2,b_3)}} \cdot \lambda(a_1,b_2 \AB{a_2}{b_3})}.\]
Property TB1 of a twisted bilinear map says that this equals
\begin{equation}
\label{eqn:intermediate3}
\BK{b_1}{\AK{a_1}{\BK{b_2}{\lambda(a_2,b_3)}} \cdot \lambda(a_1,b_2) \cdot \BK{\AB{a_1}{b_2}}{\lambda(a_1,\AB{a_2}{b_3})}}.
\end{equation}
Since \eqref{eqn:intermediate2} and \eqref{eqn:intermediate3} are equal, the claim follows.
\end{proof}

\subsection{Presentations of \texorpdfstring{$\Aut(F_n)$}{Aut(Fn)} and \texorpdfstring{$\AutFB{n}{1}$}{An1}}
\label{section:presentations}

This is the third and final section of preliminaries for the proof of Theorem \ref{maintheorem:lpresentation}.  In it, we
give presentations for the groups $\Aut(F_n)$ and $\AutFB{n}{1}$.  

We begin with $\Aut(F_n)$.  Recall that $X=\{x_1,\ldots,x_n\}$ is the standard basis for $F_n$.  The presentation
for $\Aut(F_n)$ we will use has three classes of generators.
\begin{compactitem}
\item For $\alpha \in \{1,-1\}$ and distinct $x_a, x_b \in X$, we need the automorphisms $\Mul{x_a^{\alpha}}{x_b}$ defined
in the introduction.  Recall that their characteristic properties are that
\[\Mul{x_a^{\alpha}}{x_b}(x_a^{\alpha}) = x_b x_a^{\alpha} \quad \text{and} \quad \Mul{x_a^{\alpha}}{x_b}(x_c) = x_c\]
for $x_c \in X$ with $x_c \neq x_a$.  The elements $\Mul{x_a^{\alpha}}{x_b}$ will be called
{\em transvections}.
\item For distinct $x_a,x_b \in X$, we will need the automorphisms $\Swap{a}{b}$ defined via
the formula 
\[\Swap{a}{b}(x_c) = \begin{cases}
x_b & \text{if $c = a$},\\
x_a & \text{if $c = b$},\\
x_c & \text{otherwise} \end{cases} \quad \quad \quad (1 \leq c \leq n).\]
The elements $\Swap{a}{b}$ will be called {\em swaps}.
\item For $x_a \in X$, we will need the automorphisms $\Inv{a}$ defined via the formula
\[\Inv{a}(x_b) = \begin{cases}
x_b^{-1} & \text{if $b=a$},\\
x_b & \text{otherwise} \end{cases} \quad \quad \quad (1 \leq b \leq n).\]
The elements $\Inv{a}$ will be called {\em inversions}.
\end{compactitem}
Let $S_A$ be the set consisting of the above generators.  The set $S_A$ does not contain elements
of the form $\Mul{x_a^{\alpha}}{x_b^{-1}}$, but we will frequently use $\Mul{x_a^{\alpha}}{x_b^{-1}}$
as an alternate notation for $\Mul{x_a^{\alpha}}{x_b}^{-1}$.
We then have the following theorem
of Nielsen.

\begin{table}
\begin{tabular}{p{0.95\textwidth}}
\hline
\begin{compactitem}
\item[N1.\ ] relations for the subgroup generated by inversions and swaps, a signed permutation group:
\begin{compactitem}
\item $\Inv{a}^2=1$ and $[\Inv{a},\Inv{b}]=1$,
\item
$\Swap{a}{b}^2=1$, $[\Swap{a}{b},\Swap{c}{d}]=1$, and
$\Swap{a}{b}\Swap{b}{c}\Swap{a}{b}^{-1}=\Swap{a}{c}$,
\item $\Swap{a}{b}\Inv{a}\Swap{a}{b}^{-1}=\Inv{b}$ and $[\Swap{a}{b},\Inv{c}]=1$;
\end{compactitem}
\item[N2.\ ] relations for conjugating transvections by inversions and swaps, coming from the natural action of inversions and swaps on $\{x_1,\dotsc,x_n\}$:
\begin{compactitem}
\item $\Swap{a}{b}\Mul{x_c^\gamma}{x_d}\Swap{a}{b}^{-1}=\Mul{\Swap{a}{b}(x_c^\gamma)}{\Swap{a}{b}(x_d)}$ even if $\{a,b\}\cap\{c,d\}\neq\varnothing$,
\item $\Inv{a}\Mul{x_c^\gamma}{x_d}\Inv{a}^{-1}=\Mul{\Inv{a}(x_c^\gamma)}{\Inv{a}(x_d)}$ even if $a\in\{c,d\}$;
\end{compactitem}
\item[N3.\ ]
$\Mul{x_a^{-\alpha}}{x_b}^\beta\Mul{x_b^\beta}{x_a}^\alpha\Mul{x_a^\alpha}{x_b}^{-\beta}=\Inv{b}\Swap{a}{b}$;
\item[N4.\ ] $[\Mul{x_a^\alpha}{x_b},\Mul{x_c^\gamma}{x_d}]=1$ with $a,b,c,d$ not necessarily all distinct, such that $a\neq b$, $c\neq d$, $x_a^\alpha\notin\{x_c^\gamma, x_d, x_d^{-1}\}$ and $x_c^\gamma\notin\{x_b, x_b^{-1}\}$;
\item[N5.\ ] $\Mul{x_b^\beta}{x_a}^\alpha\Mul{x_c^\gamma}{x_b}^{\beta}=\Mul{x_c^\gamma}{x_b}^{\beta}\Mul{x_b^\beta}{x_a}^\alpha\Mul{x_c^\gamma}{x_a}^\alpha$.
\end{compactitem}\\
\hline
\end{tabular}
\caption{Nielsen's relations for $\Aut(F_n)$ consist of the set $R_A$ of relations listed above.
The letters $a,b,c,d$ are elements of $\{1,\ldots,n\}$ (assumed distinct unless otherwise stated) and 
$\alpha,\beta,\gamma, \in \{1,-1\}$.}
\label{table:autfnpresentation}
\end{table}

\begin{theorem}[Nielsen~\cite{NielsenPresentation}]
\label{theorem:autfnpresentation}
The group $\Aut(F_n)$ has the presentation $\Pres{S_A}{R_A}$, where $R_A$ is given in Table~\ref{table:autfnpresentation}.
\end{theorem}

We now turn to $\AutFB{n}{1}$.  Recall that this is a subgroup of $\Aut(F_{n,1})$, where $F_{n,1}$ is the free
group on $\{x_1,\ldots,x_n,y\}$.  We will use a presentation that is due to Jensen--Wahl~\cite{JensenWahl}.  See
\cite[Theorem 5.2]{DayPutmanBirman} for a small correction to Jensen--Wahl's original statement.

\begin{theorem}[Jensen--Wahl~\cite{JensenWahl}]
\label{th:JensenWahl}
The group $\AutFB{n}{1}$ has the presentation whose generators are the union of $S_A$ with the set
\[\Set{$\Mul{x_a^\alpha}{y}$,$\Con{y}{x_a}$}{$x_a \in X$, $\alpha\in\{1,-1\}$}\]
and whose relations are those appearing in Table~\ref{table:an1presentation}.
\end{theorem}

\begin{table}
\begin{tabular}{p{0.95\textwidth}}
\hline
\begin{compactitem}
\item[Q1.\ ] Nielsen's relations among $S_A$ from Table \ref{table:autfnpresentation},
\item[Q2.\ ] Commuting relations:
\begin{itemize}
\item $[\Mul{x_a^\alpha}{y},\Mul{x_b^\beta}{y}]=1$ if $x_a^\alpha\neq x_b^\beta$,
\item $[\Mul{x_a^\alpha}{x_b},\Mul{x_c^\gamma}{y}]=1$ if $x_a^\alpha\neq x_c^\gamma$,
\item $[\Mul{x_a^\alpha}{x_b},\Con{y}{x_c}]=1$ if $c\neq a$.
\end{itemize}
\item[Q3.\ ] The obvious analogues of the N2 relations from Table \ref{table:autfnpresentation} giving
the effect of conjugating $\Con{y}{x_a}$ and $\Mul{x_a^\alpha}{y}$ by swaps and inversions,
\item[Q4.\ ] $\Mul{x_a^\alpha}{x_b}^{-\beta}\Mul{x_b^\beta}{y}\Mul{x_a^\alpha}{x_b}^\beta = \Mul{x_a^\alpha}{y}\Mul{x_b^\beta}{y}$, and
\item[Q5.\ ] $\Con{y}{x_a}^{-\alpha}\Mul{x_a^{-\alpha}}{y}\Con{y}{x_a}^\alpha=\Mul{x_a^\alpha}{y}^{-1}$.
\end{compactitem}\\
\hline
\end{tabular}
\caption{Jensen--Wahl's relations for $\AutFB{n}{1}$ consist of the relations above.  
The letters $a,b,c$ are elements of $\{1,\ldots,n\}$
(assumed distinct unless otherwise stated) and $\alpha,\beta,\gamma, \in \{1,-1\}$.}
\label{table:an1presentation}
\end{table}

\section{A finite L-presentation for \texorpdfstring{$\BKerIA{n}{1}$}{KIAn1}}
\label{section:lpresentation}

This section contains the proof of Theorem \ref{maintheorem:lpresentation}, which
asserts that $\BKerIA{n}{1}$ has a finite L-presentation.  We begin in
\S \ref{section:lpresentationstatement} by describing the generators, relations,
and endomorphisms which make up our L-presentation.  Next, in 
\S \ref{section:lpresentationextension} we construct the data needed to use
the theory of twisted bilinear maps to construct an appropriate extension
of our purported presentation for $\BKerIA{n}{1}$.  Finally, in
\S \ref{section:lpresentationproof} we prove that our presentation
is complete.

The proofs of several of our lemmas will depend on computer calculations.  These
computer calculations will be discussed in \S \ref{section:computer}.

Just like in \S \ref{section:lpresentationprelim}, we will denote
the free basis for $F_{n,1}$ by $\{x_1,\ldots,x_n,y\}$.

\subsection{Statement of L-presentation}
\label{section:lpresentationstatement}

In this section, we will describe the generators, relations, and endomorphisms that make up the finite L-presentation for 
$\BKerIA{n}{1}$ whose existence is asserted by Theorem \ref{maintheorem:lpresentation}.  To help us keep
track of the role that our symbols are playing, we will change the font for the generators of
$\BKerIA{n}{1}$ and use
\[S_K=\Set{$\kCon{y}{x_a}$, $\kCon{x_a}{y}$}{$x_a\in X$} \cup \Set{$\kMul{x_a^\alpha}{[y^\epsilon,x_b^\beta]}$}{$x_a,x_b\in X$, $x_a\neq x_b$, $\alpha,\beta,\epsilon\in\{1,-1,\}$}\]
as our generating set.  We remark that elements of the form $\kMul{x_a^\alpha}{[x_b^\beta,y^\epsilon]}$ 
are not included in $S_K$; however, we will
frequently use the symbol $\kMul{x_a^\alpha}{[x_b^\beta,y^\epsilon]}$ as a synonym for $\kMul{x_a^\alpha}{[y^\epsilon,x_b^\beta]}^{-1}$, which is
the inverse of an element of $S_K$. 

Next, we explain our endomorphisms $E_K$ for our L-presentation.  These endomorphisms are indexed by the generators for
$\AutFB{n}{1}$ given by Theorem \ref{th:JensenWahl} that do not lie in $\BKerIA{n}{1}$.  More precisely, define
\[S_Q=\Set{$\Swap{a}{b}$, $\Inv{a}$, $\Mul{x_a}{y}$, $\Mul{x_a^\alpha}{x_b}$}{$x_a, x_n \in X$, $x_a \neq x_b$, $\alpha\in\{1,-1\}$}.\]
We use $S_Q^{\pm1}$ to denote $S_Q\cup\{s^{-1}|s\in S_Q\}$ (and similarly for other sets).  We remark that
in $S_Q$, we regard $\Swap{a}{b}$ and $\Swap{b}{a}$ as being the same element.  

We now define a function $\phi\co S_Q^{\pm1}\to \End(F(S_K))$ (see Lemma 
\ref{lemma:tphiproperty} below for an elucidation of the purpose of this
definition).  By the universal property of a free group, 
to do this it is enough to give $\phi(s)(t)$ for each
choice of $s\in S_Q^{\pm1}$ and $t\in S_K$.  
Once we have given these formulas, we define $E_K$ to be the image of the map $\phi$.  
In our defining formulas, we use $x_a,x_b,x_c,\dotsc$ for elements of $X$ and $\alpha,\beta,\gamma,\epsilon,\dotsc$ for elements of $\{1,-1\}$.
Elements with distinct subscripts are assumed to be distinct unless noted.

The action of swaps and inversions through $\phi$ is by acting on the elements of $X$ indexing our generators: if $s$ is a swap or inversion, then
\[\phi(s)(\kCon{x_a}{y})=\kCon{s(x_a)}{y},\quad \phi(s)(\kCon{y}{x_a})=\kCon{y}{s(x_a)},\quad\text{and}\quad \phi(s)(\kMul{x_a^\alpha}{[y^\epsilon,x_b^\beta]})=\kMul{s(x_a^\alpha)}{[y^\epsilon,s(x_b^\beta)]}.\]
Here we interpret
\[\kCon{x_a^{-1}}{y}=\kCon{x_a}{y}\quad\text{and}\quad\kCon{y}{x_a^{-1}}=\kCon{y}{x_a}^{-1}.\]
Furthermore, for $s$ a swap or inversion we define $\phi(s^{-1})(t)=\phi(s)(t)$ for any $t\in S_K$.

If $s\in S_Q$ is of the form $\Mul{x_a^\alpha}{y}$, we define $\phi(s)(t)=\phi(s^{-1})(t)=t$ if $t\in S_K$ is anything of the form
\[\kCon{x_a}{y},\quad \kCon{x_b}{y},\quad \kMul{x_a^{-\alpha}}{[y^\epsilon,x_b^\beta]},\quad \text{or}\quad \kMul{x_b^{\beta}}{[y^\epsilon,x_c^\gamma]}.\]
If $s\in S_Q$ is of the form $\Mul{x_a^\alpha}{x_b}$, we define $\phi(s)(t)=\phi(s^{-1})(t)=t$ if $t\in S_K$ is anything of the form
\[\kCon{x_c}{y},\quad \kCon{y}{x_b},\quad \kCon{y}{x_c},\quad \kMul{x_a^{-\alpha}}{[y^\epsilon,x_b^\beta]},\quad \kMul{x_a^{-\alpha}}{[y^\epsilon,x_c^\gamma]},\quad \kMul{x_c^\gamma}{[y^\epsilon,x_b^\beta]},\quad \text{or}\quad \kMul{x_c^\gamma}{[y^\epsilon,x_d^\delta]}.\]
The other cases for $\phi(s)(t)\in F(S_K)$, for $s\in S_Q^{\pm1}$ and $t\in S_K$, are given in Table~\ref{table:definephi}.

\begin{table}
\tiny
\centering
\begin{tabular}{cc|c}
$s\in S_Q^{\pm1}$ & $t\in S_K$ & $\phi(s)(t)$ \\
\hline
$\Mul{x_a}{y}^\epsilon$ & $\kCon{y}{x_b}$ & $\kCon{y}{x_b}\kMulcomm{x_a}{y^{-\epsilon}}{x_b^{-1}}$ \\
 & $\kCon{y}{x_a}$ & $\kCon{x_a}{y}^\epsilon\kCon{y}{x_a}$ \\
 & $\kMulcomm{x_a}{y^\zeta}{x_b^\beta}$ & $\kCon{x_a}{y}^\epsilon\kMulcomm{x_a}{y^\zeta}{x_b^\beta}\kCon{x_a}{y}^{-\epsilon}$ \\
 & $\kMulcomm{x_b^\beta}{y^\zeta}{x_a}$ & $\kCon{x_a}{y}^\epsilon\kMulcomm{x_b^\beta}{y^\zeta}{x_a}\kCon{x_a}{y}^{-\epsilon}$\\
\hline
$\Mul{x_a^\alpha}{x_b}^\beta$ & $\kCon{y}{x_a}$ & $(\kCon{y}{x_a}^\alpha\kCon{y}{x_b}^\beta)^\alpha$ \\
 & $\kCon{x_a}{y}$ & $\kCon{x_a}{y}\kMulcomm{x_a^\alpha}{x_b^{-\beta}}{y}$ \\
 & $\kCon{x_b}{y}$ & $\kCon{x_b}{y}\kMulcomm{x_a^\alpha}{x_b^{-\beta}}{y^{-1}}$\\
 & $\kMulcomm{x_a^\alpha}{y^\epsilon}{x_c^\gamma}$ & $\kCon{y}{x_c}^\gamma\kMulcomm{x_a^\alpha}{y^\epsilon}{x_b^{-\beta}}\kCon{y}{x_c}^{-\gamma}\kMulcomm{x_a^\alpha}{y^\epsilon}{x_c^\gamma}\kMulcomm{x_a^\alpha}{x_b^{-\beta}}{y^\epsilon}$\\
 & $\kMulcomm{x_b^\beta}{y^\epsilon}{x_c^\gamma}$ & $\kMulcomm{x_b^\beta}{y^\epsilon}{x_c^\gamma}\kMulcomm{x_a^\alpha}{y^\epsilon}{x_b^{-\beta}}\kMulcomm{x_a^\alpha}{x_c^\gamma}{y^\epsilon}\kCon{y^\epsilon}{x_c^\gamma}\kMulcomm{x_a^\alpha}{x_b^{-\beta}}{y^\epsilon}\kCon{y^\epsilon}{x_c}^{-\gamma}$\\
 & $\kMulcomm{x_b^{-\beta}}{y^\epsilon}{x_c^\gamma}$ & $\kMulcomm{x_b^{-\beta}}{y^\epsilon}{x_c^\gamma}\kMulcomm{x_a^\alpha}{y^\epsilon}{x_c^\gamma}$ \\
 & $\kMulcomm{x_c^\gamma}{y^\epsilon}{x_a^\alpha}$ & $\kCon{y}{x_a}^\alpha\kMulcomm{x_c^\gamma}{y^\epsilon}{x_b^\beta}\kCon{y}{x_a}^{-\alpha}\kMulcomm{x_c^\gamma}{y^\epsilon}{x_a^\alpha}$ \\
 & $\kMulcomm{x_c^\gamma}{y^\epsilon}{x_a^{-\alpha}}$ & $\kCon{y}{x_b}^{-\beta}\kMulcomm{x_c^\gamma}{y^\epsilon}{x_a^{-\alpha}}\kMulcomm{x_c^\gamma}{x_b^\beta}{y^\epsilon}\kCon{y}{x_b^\beta}$ \\
 & $\kMulcomm{x_a^\alpha}{y^\epsilon}{x_b^\beta}$ & $\kMulcomm{x_a^\alpha}{x_b^{-\beta}}{y^\epsilon}$\\
 & $\kMulcomm{x_a^\alpha}{y^\epsilon}{x_b^{-\beta}}$ & $\kCon{y}{x_b}^{-\beta}\kMulcomm{x_a^\alpha}{y^\epsilon}{x_b^{-\beta}}\kCon{y}{x_b}^\beta$ \\
 & $\kMulcomm{x_b^\beta}{y^\epsilon}{x_a^\alpha}$  & $\kMulcomm{x_b^{-\beta}}{y^{-\epsilon}}{x_a^\alpha}\kCon{x_b}{y}^\epsilon\kMulcomm{x_a^\alpha}{y^\epsilon}{x_b^{-\beta}}\kCon{x_a}{y}^{-\epsilon}$ \\
 & $\kMulcomm{x_b^{-\beta}}{y^\epsilon}{x_a^\alpha}$ & $\kCon{y}{x_a}^\alpha\kCon{y}{x_b}^\beta\kCon{x_b}{y}^{-\epsilon}\kMulcomm{x_a^\alpha}{x_b^{-\beta}}{y^\epsilon}\kCon{y}{x_b}^{-\beta}\kMulcomm{x_b^{-\beta}}{x_a^{-\alpha}}{y^\epsilon}\kCon{y}{x_a}^{-\alpha}\kCon{x_a}{y}^\epsilon$ \\
 & $\kMulcomm{x_b^\beta}{y^\epsilon}{x_a^{-\alpha}}$ & $\kCon{y}{x_b}^{-\beta}\kCon{y}{x_a}^{-\alpha}\kCon{x_a}{y}^\epsilon\kMulcomm{x_a^\alpha}{x_b^{-\beta}}{y^\epsilon}\kCon{x_b}{y}^{-\epsilon}\kMulcomm{x_b^{-\beta}}{x_a^\alpha}{y^{-\epsilon}}\kCon{y}{x_a}^\alpha\kCon{y}{x_b}^\beta$ \\
 & $\kMulcomm{x_b^{-\beta}}{y^\epsilon}{x_a^{-\alpha}}$ & $\kCon{y}{x_b}^{-\beta}\kCon{y}{x_a}^{-\alpha}\kCon{x_a}{y}^{-\epsilon}\kCon{y}{x_a}^\alpha\kMulcomm{x_b^{-\beta}}{y^\epsilon}{x_a^{-\alpha}}\kCon{y}{x_b}^\beta\kMulcomm{x_a^\alpha}{y^\epsilon}{x_b^{-\beta}}\kCon{x_b}{y}$\\

\end{tabular}
\caption{The defining formulas for $\phi$.  The letters $a,b,c$ are distinct elements of $\{1,\ldots,n\}$
$\alpha,\beta,\gamma,\zeta \in \{1,-1\}$.}
\label{table:definephi}
\end{table}

The key property of the map $\phi$ is as follows.

\begin{lemma}
\label{lemma:tphiproperty}
Let $\Psi\co F(S_K) \rightarrow \BKerIA{n}{1}$ be the natural surjection.  Then
regarding $S_Q$ as a subset of $\AutFB{n}{1}$, we have
\[\Psi(\phi(s)(t)) = s \Psi(t) s^{-1} \quad \quad (s \in S_Q, t \in S_K).\]
\end{lemma}
\begin{proof}
This is a computer calculation which is described in Lemma \ref{lemma:computerphiproperty}
below.
\end{proof}

We can now give a statement of our L-presentation.  The following theorem (which will be proven in \S \ref{section:lpresentationproof}) is a
more precise version of Theorem \ref{maintheorem:lpresentation}.

\begin{theorem}\label{theorem:preciselpres}
Let $S_K$ and $S_Q$ and $\phi$ be as above and let $R_K^0$ be the set of relations in Table \ref{table:basicrelsbkerian}.
Then the group $\BKerIA{n}{1}$ has the finite L-presentation $\BKerIA{n}{1}=\LPres{S_K}{R_K^0}{\phi(S_Q^{\pm1})}$.
\end{theorem}

\begin{table}
\begin{tabular}{p{0.95\textwidth}}
\hline
\begin{compactitem}
\item[R1.\ ] $[\kCon{x_a}{y},\kCon{x_b}{y}]=1$;
\item[R2.\ ] $[\kMulcomm{x_a^\alpha}{y^\epsilon}{x_c^\gamma},\kMulcomm{x_b^\beta}{y}{x_d^\delta}]=1$,
possibly with $x_a^\alpha=x_b^{-\beta}$ or $x_c=x_d$ (or both), as long as $x_a^\alpha\neq x_b^\beta$, $x_a\neq x_d$ and $x_b\neq x_c$;
\item[R3.\ ] $[\kCon{x_a}{y},\kMulcomm{x_b^\beta}{y^\epsilon}{x_c^\gamma}]=1$;
\item[R4.\ ] $\kCon{y}{x_b}^{-\beta}\kMulcomm{x_a^\alpha}{y^\epsilon}{x_b^\beta}\kCon{y}{x_b}^{\beta}=\kMulcomm{x_a^\alpha}{x_b^{-\beta}}{y^\epsilon}$;
\item[R5.\ ] $\kCon{x_b}{y}^{-\epsilon}\kMulcomm{x_a^\alpha}{y^\epsilon}{x_b^\beta}\kCon{x_b}{y}^{\epsilon}=\kMulcomm{x_a^\alpha}{x_b^\beta}{y^{-\epsilon}}$;
\item[R6.\ ] $\kCon{x_a}{y}^\epsilon\kMulcomm{x_a^\alpha}{y^\epsilon}{x_b^\beta}\kCon{x_a}{y}^{-\epsilon}=\kMulcomm{x_a^\alpha}{x_b^\beta}{y^{-\epsilon}}$;
\item[R7.\ ] $\kMulcomm{x_a^\alpha}{y^\epsilon}{x_b^\beta}\kMulcomm{x_a^{-\alpha}}{y^\epsilon}{x_b^\beta}=\kCon{y}{x_b}^{\beta}\kCon{x_a}{y}^{-\epsilon}\kCon{y}{x_b}^{-\beta}\kCon{x_a}{y}^{\epsilon}$;
\item[R8.\ ] $\kMulcomm{x_b^\beta}{y^{-\epsilon}}{x_c^\gamma}\kMulcomm{x_a^\alpha}{y^\epsilon}{x_b^\beta}\kMulcomm{x_b^\beta}{x_c^\gamma}{y^{-\epsilon}}
=\kMulcomm{x_a^\alpha}{x_c^\gamma}{y^{-\epsilon}}\kMulcomm{x_a^\alpha}{y^\epsilon}{x_b^\beta}\kMulcomm{x_a^\alpha}{x_c^\gamma}{y^\epsilon}$;
\item[R9.\ ] 
$\kCon{x_b}{y}^{-\epsilon}\kCon{y}{x_c}^\gamma\kMulcomm{x_a^\alpha}{y^\epsilon}{x_b^\beta}\kCon{y}{x_c}^{-\gamma}\kCon{x_b}{y}^\epsilon$

$\quad\quad=\kMulcomm{x_a^\alpha}{x_b^\beta}{y^{-\epsilon}}\kCon{y}{x_c}^\gamma\kMulcomm{x_a^\alpha}{y^\epsilon}{x_b^\beta}\kCon{y}{x_c}^{-\gamma}\kMulcomm{x_a^\alpha}{y^\epsilon}{x_c^\gamma}\kMulcomm{x_a^\alpha}{x_b^\beta}{y^\epsilon}\kMulcomm{x_a^\alpha}{x_c^\gamma}{y^\epsilon}$;
\item[R10.\ ]
$\kCon{x_c}{y}^{-\epsilon}\kCon{y}{x_c}^\gamma\kMulcomm{x_a^{\alpha}}{y^{\epsilon}}{x_b^\beta}\kCon{y}{x_c}^{-\gamma}\kCon{x_c}{y}^{\epsilon}$

$\quad\quad=\kMulcomm{x_a^{\alpha}}{y^{-\epsilon}}{x_b^\beta}\kMulcomm{x_a^{\alpha}}{x_c^\gamma}{y^{-\epsilon}}\kCon{y}{x_c}^\gamma\kMulcomm{x_a^{\alpha}}{x_b^\beta}{y^{-\epsilon}}\kCon{y}{x_c}^{-\gamma}\kMulcomm{x_a^{\alpha}}{y^{\epsilon}}{x_b^\beta}\kMulcomm{x_a^{\alpha}}{y^{-\epsilon}}{x_c^\gamma}$.
\end{compactitem}\\
\hline
\end{tabular}
\caption{The relations $R_K^0$ for the $L$-presentation of $\BKerIA{n}{1}$.  The letters $a,b,c,d$ are elements of $\{1,\ldots,n\}$
(assumed distinct unless otherwise stated) and $\alpha,\beta,\gamma,\delta,\epsilon \in \{1,-1\}$.}
\label{table:basicrelsbkerian}
\end{table}

\subsection{Constructing the extension}
\label{section:lpresentationextension}

Let $\Gamma_n = \LPres{S_K}{R_K^0}{\phi(S_Q^{\pm1})}$ 
be the group with the presentation described in Theorem \ref{theorem:preciselpres}.  There
is thus a surjection $\Psi\co \Gamma_n \rightarrow \BKerIA{n}{1}$ which Theorem \ref{theorem:preciselpres} claims
is an isomorphism.  Lemma \ref{lemma:relatingkernels} says that there exists a short exact sequence
\begin{equation}
\label{eqn:knownextension}
1 \longrightarrow \BKerIA{n}{1} \longrightarrow \AutFB{n}{1} \stackrel{\rho}{\longrightarrow} \zasdr \longrightarrow 1
\end{equation}
together with homomorphisms $\iota_1\co \Aut(F_n) \rightarrow \AutFB{n}{1}$ and
$\iota_2\co \Z^n \rightarrow \AutFB{n}{1}$ such that $\rho \circ \iota_1 = \text{id}$
and $\rho \circ \iota_2 = \text{id}$.  The purpose of this section is to construct the data needed to
apply Theorem \ref{theorem:twistedextend} and deduce that there exists a similar extension involving
$\Gamma_n$ instead of $\BKerIA{n}{1}$.

For $f \in \Aut(F_n)$ and $z \in \Z^n$, define homomorphisms $\overline{\alpha}_f\co \BKerIA{n}{1} \rightarrow \BKerIA{n}{1}$
and $\overline{\beta}_z\co \BKerIA{n}{1} \rightarrow \BKerIA{n}{1}$ via the formulas
\[\overline{\alpha}_f(x) = \iota_1(f) x \iota_1(f)^{-1} \quad \text{and} \quad \overline{\beta}_z(x) = \iota_2(z) x \iota_2(z)^{-1} \quad \quad (x \in \BKerIA{n}{1}).\]
These define actions of $\Aut(F_n)$ and $\Z^n$ on $\BKerIA{n}{1}$.  Using the construction described in
Example \ref{example:groupextend}, we obtain from \eqref{eqn:knownextension} a twisted bilinear map 
$\overline{\lambda}\co \Aut(F_n) \times \Z^n \rightarrow \BKerIA{n}{1}$.  We must lift all of this data
to $\Gamma_n$.  This is accomplished in the following three lemmas.  For a set
$S$, let $S^{\ast}$ denote the free monoid on $S$, so $S^{\ast}$ consists of
words in $S$.

\begin{lemma}
\label{lemma:action1}
There exists an action of $\Aut(F_n)$ on $\Gamma_n$ with the following property.  For $f \in \Aut(F_n)$, denote
by $\alpha_f\co\Gamma_n \rightarrow \Gamma_n$ the associated automorphism.  Then
\[\Psi(\alpha_f(x)) = \overline{\alpha}_f(\Psi(x)) \quad \quad (f \in \Aut(F_n), x \in \Gamma_n).\]
\end{lemma}
\begin{proof}
Let $\Aut(F_n) = \Pres{S_A}{R_A}$ be the presentation given by Theorem 
\ref{theorem:autfnpresentation}.  We have $S_A \subset S_Q$, so the
map $\phi\co S_Q^{\pm 1} \rightarrow \End(F(S_K))$ used in the construction
of the L-presentation for $\Gamma_n$ restricts to a set map
$S_A^{\pm 1} \rightarrow \End(F(S_K))$.  By the definition of an L-presentation,
the image of this set map preserves the relations between elements of $S_K$
that make up $\Gamma_n$, so we get a set map $S_A^{\pm 1} \rightarrow \End(\Gamma_n)$.
By the universal property of the free monoid, this induces a monoid homomorphism
$\zeta\co(S_A^{\pm 1})^{\ast} \rightarrow \End(\Gamma_n)$.  Computer calculations
described in Lemma \ref{lemma:computeraction11} below show that
$\zeta(s)(\zeta(s^{-1})(t)) = t$ for all $s \in S_A^{\pm 1}$ and $t \in S_K$, which
implies that the image of $\zeta$ is contained in $\Aut(\Gamma_n)$ and that 
$\zeta$ descends to a group homomorphism $\eta\co F(S_A) \rightarrow \Aut(\Gamma_n)$.
Further computer calculations described in Lemma \ref{lemma:computeraction12} below
show that $\eta(r)(t) = t$ for all $r \in R_A$ and all $t \in S_K$.  This
implies that $\eta$ descends to a group homomorphism
$\Aut(F_n) \rightarrow \Aut(\Gamma_n)$.  This is the desired action; the claimed
naturality property follows from Lemma \ref{lemma:tphiproperty}.
\end{proof}

\begin{lemma}
\label{lemma:action2}
There exists an action of $\Z^n$ on $\Gamma_n$ with the following property.  For $z \in \Z^n$, denote
by $\beta_z\co \Gamma_n \rightarrow \Gamma_n$ the associated automorphism.  Then
\[\Psi(\beta_z(x)) = \overline{\beta}_z(\Psi(x)) \quad \quad (z \in \Z^n, x \in \Gamma_n).\]
\end{lemma}
\begin{proof}
Set $S_Z = \{\Mul{x_1}{y},\ldots,\Mul{x_n}{y}\}$ and 
$R_Z = \Set{$[\Mul{x_i}{y}, \Mul{x_j}{y}]$}{$1 \leq i < j \leq n$}$, so
$S_Z \subset S_Q$ and $\Z^n = \Pres{S_Z}{R_Z}$.  Just like in the proof
of Lemma \ref{lemma:action1}, the 
map $\phi\co S_Q^{\pm 1} \rightarrow \End(F(S_K))$ restricts to a set map
$S_Z^{\pm 1} \rightarrow \End(F(S_K))$ which induces a monoid homomorphism
$\zeta\co(S_Z^{\pm 1})^{\ast} \rightarrow \End(\Gamma_n)$.  Computer calculations
described in Lemma \ref{lemma:computeraction21} below show that
$\zeta(s)(\zeta(s^{-1})(t)) = t$ for all $s \in S_Z^{\pm 1}$ and $t \in S_K$, 
so $\zeta$ induces a group homomorphism
$\eta\co F(S_Z) \rightarrow \Aut(\Gamma_n)$.
Further computer calculations described in Lemma \ref{lemma:computeraction22} below
show that $\eta(r)(t) = t$ for all $r \in R_Z$ and all $t \in S_K$.  This
implies that $\eta$ descends to a group homomorphism
$\Z^n \rightarrow \Aut(\Gamma_n)$.  This is the desired action; the claimed
naturality property follows from Lemma \ref{lemma:tphiproperty}.
\end{proof}

\begin{lemma}
\label{lemma:twistedbilinear}
With respect to the action of $\Aut(F_n)$ on $\Z^n$ introduced in \S \ref{section:relatingkernels} and
the actions of $\Z^n$ and $\Aut(F_n)$ on $\Gamma_n$ given by Lemmas \ref{lemma:action1} and \ref{lemma:action2}, there
exists a twisted bilinear map $\lambda\co \Aut(F_n) \times \Z^n \rightarrow \Gamma_n$ such that
\[\Psi(\lambda(f,z)) = \overline{\lambda}(f,z) \quad \quad (f \in \Aut(F_n), z \in \Z^n).\]
\end{lemma}
\begin{proof}
Let $\alpha_f$ and $\beta_z$ be as in Lemmas \ref{lemma:action1} and \ref{lemma:action2}, respectively.
Let $\Aut(F_n) = \Pres{S_A}{R_A}$ be the presentation given by Theorem
\ref{theorem:autfnpresentation}.  Also, let $S_Z = \{\Mul{x_1}{y},\ldots,\Mul{x_n}{y}\}$ and
$R_Z = \Set{$[\Mul{x_i}{y}, \Mul{x_j}{y}]$}{$1 \leq i < j \leq n$}$, so
$\Z^n = \Pres{S_Z}{R_Z}$.  We claim that it is enough to construct a twisted bilinear map
$\lambda\co \Aut(F_n) \times \Z^n \rightarrow \Gamma_n$ such that
\begin{equation}
\label{eqn:desired}
\Psi(\lambda(f,z)) = \overline{\lambda}(f,z) \quad \quad (f \in S_A^{\pm 1}, z \in S_Z^{\pm 1}).
\end{equation}
Indeed, the axioms of a twisted bilinear map show that $\lambda$ is determined by its values on
generators: property TB2 says that
$\lambda(a_1 a_2,b) = \AK{a_1}{\lambda(a_2,b)} \cdot \lambda(a_1,\AB{a_2}{b})$ for all
$a_1,a_2 \in \Aut(F_n)$ and $b \in \Z^n$, so the values of $\lambda$ are determined
by the values of $\lambda(f,z)$ for $f \in S_A^{\pm 1}$ and $z \in \Z^n$, and then
property TB1 says that $\lambda(a,b_1 b_2) = \lambda(a,b_1) \cdot \BK{\AB{a}{b_1}}{\lambda(a,b_2)}$
for all $a \in \Aut(F_n)$ and $b_1,b_2 \in \Z^n$, so the values of $\lambda$ are
determined by the values of $\lambda(f,z)$ for $f \in S_A^{\pm 1}$ and $z \in S_Z^{\pm 1}$.  
An analogous fact holds for $\overline{\lambda}$, whence the claim.

\begin{table}
\begin{center}
\begin{tabular}{c|c|c}
$f\in S_A^{\pm1}$ & $z\in S_Z^{\pm1}$ & $\lambda(f,z)\in F(S_K)$ \\
\hline
$I_a^{\pm1}$ & $\Mul{x_a}{y}^\epsilon$ & $\kCon{x_a}{y}^\epsilon$ \\
$\Mul{x_a}{x_b}^\beta$ & $\Mul{x_a}{y}^\epsilon$ & $\kMulcomm{x_a}{y^{-\epsilon}}{x_b^{-\beta}}$ \\
$\Mul{x_a}{x_b}$ & $\Mul{x_b}{y}^\epsilon$ & $\kMulcomm{x_a}{y^{\epsilon}}{x_b^{-1}}$ \\
$\Mul{x_a^{-1}}{x_b}$ & $\Mul{x_b}{y}^\epsilon$ & $(\kMulcomm{x_a^{-1}}{y}{x_b^{-1}}\kCon{x_a}{y}^{-1})^\epsilon$ \\
$\Mul{x_a^{-1}}{x_b}^{-1}$ & $\Mul{x_b}{y}^\epsilon$ & $\kCon{x_a}{y}^\epsilon$ \\
\end{tabular}
\end{center}
\caption{The effect of $\lambda(\cdot,\cdot)$ on generators.  For $f \in S_A^{\pm 1}$ and $z \in S_Z^{\pm 1}$ such
that there is no entry in the above table, we have $\lambda(f,z)=1$.}
\label{table:lambdadef}
\end{table}

We will construct $\lambda$ such that $\lambda(f,z)$ is as in Table \ref{table:lambdadef} for $f \in S_A^{\pm 1}$
and $z \in S_Z^{\pm 1}$.  It is easy to check that these values satisfy \eqref{eqn:desired}.  We will do this
in four steps.  For a set $S$, let $S^{\ast}$ be the free monoid on $S$, so $S^{\ast}$ consists of
words in $S$.
\begin{compactitem}
\item First, for $f \in S_A^{\pm 1}$ we will use the ``expansion rule'' TB1 to
construct a map $\tlambda_1(f,\cdot)$ from $(S_Z^{\pm 1})^{\ast}$ to $\Gamma_n$ 
with $\tlambda_1(f,z)$ equal to the value of $\lambda(f,z)$ from Table \ref{table:lambdadef} 
for $z \in S_Z^{\pm 1}$.
\item Next, we will show that $\tlambda_1(f,\cdot)$ descends to a map $\lambda_1(f,\cdot)$ from $\Z^n$ to $\Gamma_n$.
\item Next, for $z \in \Z^n$ we will use the ``expansion rule'' TB2 to
construct a map $\tlambda_2(\cdot,z)$ from $(S_A^{\pm 1})^{\ast}$
to $\Gamma_n$ with $\tlambda_2(f,z) = \tlambda_1(f,z)$ for $f \in S_A^{\pm 1}$.
\item Finally, we will show that $\tlambda_2(f,\cdot)$ descends to a map $\lambda_2(f,\cdot)$ from $\Aut(F_n)$ to
$\Gamma_n$.
\end{compactitem}
The desired twisted bilinear map will then be defined by $\lambda(f,z) = \lambda_2(f,z)$.  It will follow
from the various intermediate steps in our construction that $\lambda(\cdot,\cdot)$ is a twisted bilinear map.

As notation, for $w \in (S_A^{\pm 1})^{\ast}$, let $\hw$ denote the image of $w$ in $\Aut(F_n)$.  Similarly,
for $w \in (S_Z^{\pm 1})^{\ast}$, let $\hw$ denote the image of $w$ in $\Z^n$.

We now construct $\tlambda_1$.  For $f \in S_A^{\pm 1}$ and $w \in (S_Z^{\pm 1})^{\ast}$, we define
$\tlambda_1(f,w) \in \Gamma_n$ by induction on the length of $w$.  If $w = 1$ (i.e.\ $w$ has length $0$), then
we define $\tlambda_1(f,w) = 1$.  If $w \in S_Z^{\pm 1}$ (i.e.\ $w$ has length $1$), then we
define $\tlambda_1(f,w)$ to be the value of $\lambda(f,w)$ from Table \ref{table:lambdadef}.  Finally, 
if $w$ has length at least $2$ and $\tlambda_1(f,\cdot)$ has been defined for all shorter words, then
write $w = s w'$ with $s \in S_Z^{\pm 1}$ and define
\[\tlambda_1(f,w) = \tlambda_1(f,s) \cdot \BK{\AB{\hf}{\hs}}{\tlambda_1(f,w')}.\]
This formula should remind the reader of property TB1 from the definition of a twisted bilinear map, as should
the following claim.

\begin{claimsb}
\label{claim:one}
$\tlambda_1(f,w w') = \tlambda_1(f,w) \cdot \BK{\AB{\hf}{\hw}}{\tlambda_1(f,w')}$ 
for $f \in S_A^{\pm 1}$ and $w,w' \in (S_Z^{\pm 1})^{\ast}$.
\end{claimsb}
\begin{proof}[Proof of claim]
The proof is by induction on the length of $w$.  For $w$ of length $0$, this is trivial, and
for $w$ of length $1$, it holds by definition.  Now assume that $w$ has length at least $2$
and that the desired formula holds whenever $w$ has smaller length.  Write $w = w_1 w_2$, where
$w_1$ and $w_2$ are shorter words than $w$.  Applying our inductive hypothesis twice, we see that
\begin{align*}
\tlambda_1(f,w w') &= \tlambda_1(f,w_1 w_2 w') = \tlambda_1(f,w_1) \cdot \BK{\AB{\hf}{\hw_1}}{\tlambda_1(f,w_2 w')} \\
&= \tlambda_1(f,w_1) \cdot \BK{\AB{\hf}{\hw_1}}{\tlambda_1(f,w_2) \cdot \BK{\AB{\hf}{\hw_2}}{\tlambda_1(f,w')}} \\
&= \tlambda_1(f,w_1) \cdot \BK{\AB{\hf}{\hw_1}}{\tlambda_1(f,w_2)} \cdot \BK{\AB{\hf}{\widehat{w_1 w_2}}}{\tlambda_1(f,w')}.
\end{align*}
Applying our inductive hypothesis to the first two terms, we see that this equals
\[\tlambda_1(f,w_1 w_2) \cdot \BK{\AB{\hf}{\widehat{w_1 w_2}}}{\tlambda_1(f,w')} = \tlambda_1(f,w) \cdot \BK{\AB{\hf}{\hw}}{\tlambda_1(f,w')}.\qedhere\]
\end{proof}

\begin{claimsb}
\label{claim:two}
For $w,w' \in (S_Z^{\pm 1})^{\ast}$ with $\hw = \hw' \in \Z^n$, we have $\tlambda_1(f,w) = \tlambda_1(f,w')$ for 
$f \in S_A^{\pm 1}$.
\end{claimsb}
\begin{proof}[Proof of claim]
Recall that $\Z^n = \GroupPres{S_Z}{R_Z}$.
Define $R_Z' = R_Z \cup \Set{$s s^{-1}$}{$s \in S_Z^{\pm 1}$} \subset (S_Z^{\pm 1})^{\ast}$.  Since 
any two elements of $(S_Z^{\pm 1})^{\ast}$ that map to the same element of $\Z^n$ must differ by
a sequence of insertions and deletions of elements of $R_Z'$, we can assume without loss of generality
that $w = u v$ and $w' = u r v$ for some $u,v \in (S_Z^{\pm 1})^{\ast}$ and $r \in R_Z'$.  A computer
calculation described in Lemma \ref{lemma:computerrelationsz} below shows that $\tlambda_1(f,r) = 1$.  
We now apply Claim \ref{claim:one} several times to deduce that
\begin{align*}
\tlambda_1(f,w') &= \tlambda_1(f,u r v) = \tlambda_1(f,u) \cdot \BK{\AB{\hf}{\hu}}{\tlambda_1(f,r)} \cdot \BK{\AB{\hf}{\widehat{ur}}}{\tlambda_1(f,v)}\\
&= \tlambda_1(f,u) \cdot \BK{\AB{\hf}{\widehat{u}}}{\tlambda_1(f,v)}
= \tlambda_1(f,uv) = \tlambda_1(f,w). \qedhere
\end{align*}
\end{proof}

For $f \in S_A^{\pm 1}$, Claim \ref{claim:two} implies that the map $\tlambda_1(f,\cdot)$ from $(S_Z^{\pm 1})^{\ast}$
to $\Gamma_n$ descends to a map $\lambda_1(f,\cdot)$ from $\Z^n$ to $\Gamma_n$.  Claim \ref{claim:one} implies that
$\lambda_1(f,\cdot)$ satisfies a version of condition TB1 from the definition of a twisted bilinear
map, namely that $\lambda_1(f,z_1 z_2) = \lambda_1(f,z_1) \cdot \BK{\AB{\hf}{z_1}}{\lambda_1(f,z_2)}$ for
all $z_1,z_2 \in \Z^n$.  Our next claim is a version of condition TB3.  We remark that the condition
$f \in S_A$ in it is not a typo; we will extend it to $f \in S_A^{\pm 1}$ later.

\begin{claimsb}
\label{claim:three}
$\lambda_1(f,z) \cdot \BK{\AB{\hf}{z}}{\AK{\hf}{k}} \cdot \lambda_1(f,z)^{-1} = \AK{\hf}{\BK{z}{k}}$
for $f \in S_A$, $z \in \Z^n$, and $k \in \Gamma_n$.
\end{claimsb}
\begin{proof}[Proof of claim]
Let $w \in (S_Z^{\pm 1})^{\ast}$ satisfy $\hw = z$.  The proof is by induction on the length of $w$.  
For $w$ of length $0$, the claim is trivial.  For $w$ of length $1$, there are two cases.  
For $w \in S_Z$, the claim follows from a computer calculation
described below in Lemma \ref{lemma:computertb3}.  For $w = v^{-1}$ with $v \in S_Z$, Claim \ref{claim:one}
implies that
\[1 = \tlambda_1(f,v^{-1}v) = \tlambda_1(f,v^{-1}) \cdot \BK{\AB{\hf}{\hv^{-1}}}{\tlambda_1(f,v)},\]
so $\tlambda_1(f,v^{-1}) = \BK{\AB{\hf}{\hv^{-1}}}{\tlambda_1(f,v)^{-1}}$.  Our goal is to show that
\[\tlambda_1(f,v^{-1}) \cdot \BK{\AB{\hf}{\hv^{-1}}}{\AK{\hf}{k}} \cdot \tlambda_1(f,v^{-1})^{-1} = \AK{\hf}{\BK{\hv^{-1}}{k}}.\]
Plugging in our formula for $\tlambda_1(f,v^{-1})$, we see that this is equivalent to showing that
\[\BK{\AB{\hf}{\hv^{-1}}}{\tlambda_1(f,v)^{-1} \cdot \AK{\hf}{k} \cdot \tlambda_1(f,v)} = \AK{\hf}{\BK{\hv^{-1}}{k}}.\]
Manipulating this a bit, we see that it is equivalent to showing that
\[\AK{\hf}{k} = \tlambda_1(f,v) \cdot \BK{\AB{\hf}{\hv}}{\AK{\hf}{\BK{\hv^{-1}}{k}}} \cdot \tlambda_1(f,v)^{-1}.\]
Using the already proven case $w=v$ of the claim, the right hand side equals
\[\AK{\hf}{\BK{\hv}{\BK{\hv^{-1}}{k}}} = \AK{\hf}{k},\]
as desired.

Now assume that $w$ has length at least $2$ and that the claim is true for all shorter words.  Write
$w = w_1 w_2$, where $w_1$ and $w_2$ are shorter words than $w$.  Applying Claim \ref{claim:one}, we see
that $\tlambda_1(f,w_1 w_2) \cdot \BK{\AB{\hf}{\widehat{w_1 w_2}}}{\AK{\hf}{k}} \cdot \tlambda_1(f,w_1 w_2)^{-1}$
equals
\begin{equation}
\label{eqn:claim3intermediate}
\tlambda_1(f,w_1) \cdot \BK{\AB{\hf}{\hw_1}}{\tlambda_1(f,w_2) \cdot \BK{\AB{\hf}{\hw_2}}{\AK{\hf}{k}} \cdot \tlambda_1(f,w_2)^{-1}} \cdot \tlambda_1(f,w_1)^{-1}.
\end{equation}
Our inductive hypothesis implies that
\[\tlambda_1(f,w_2) \cdot \BK{\AB{\hf}{\hw_2}}{\AK{\hf}{k}} \cdot \tlambda_1(f,w_2)^{-1} = \AK{\hf}{\BK{\hw_2}{k}}.\]
Thus \eqref{eqn:claim3intermediate} equals
\[\tlambda_1(f,w_1) \cdot \BK{\AB{\hf}{\hw_1}}{\AK{\hf}{\BK{\hw_2}{k}}} \cdot \tlambda_1(f,w_1)^{-1}.\]
Another application of our inductive hypothesis shows that this equals
\[\AK{\hf}{\BK{\hw_1}{\BK{\hw_2}{k}}} = \AK{\hf}{\BK{\widehat{w_1 w_2}}{k}}. \qedhere\]
\end{proof}

We now construct $\tlambda_2$.  For $w \in (S_A^{\pm 1})^{\ast}$ and $z \in \Z^n$, we define
$\tlambda_2(w,z) \in \Gamma_n$ by induction on the length of $w$.  If $w = 1$ (i.e.\ $w$ has length $0$), then
we define $\tlambda_2(w,z) = 1$.  If $w \in S_A^{\pm 1}$ (i.e.\ $w$ has length $1$), then we
define $\tlambda_2(w,z) = \lambda_1(w,z)$.  Finally,
if $w$ has length at least $2$ and $\tlambda_2(\cdot,z)$ has been defined for all shorter words, then
write $w = s w'$ with $s \in S_A^{\pm 1}$ and define
\[\tlambda_2(w,z) = \AK{\hs}{\tlambda_2(w',z)} \cdot \tlambda_2(a_1,\AB{\hw'}{z}).\]
This formula should remind the reader of property TB2 from the definition of a twisted bilinear map, as
should the following claim.

\begin{claimsb}
\label{claim:four}
$\tlambda_2(w w',z) = \AK{\hw}{\tlambda_2(w',z)} \cdot \tlambda_2(w,\AB{\hw'}{z})$ for
$w,w' \in (S_A^{\pm 1})^{\ast}$ and $z \in \Z^n$.
\end{claimsb}
\begin{proof}[Proof of claim]
This can be proved by induction on the length of $w$ just like Claim \ref{claim:one}.  The details are left
to the reader.
\end{proof}

The reader might expect at this point that we would prove an analogue of Claim \ref{claim:two} and thus
show that $\tlambda_2$ descends to a map $\lambda_2\co \Aut(F_n) \times \Z^n \rightarrow \Gamma_n$.  However,
before we can do this we must prove two preliminary results.  The first extends Claim \ref{claim:three}
to show that $\tlambda_2$ satisfies a version of condition TB3.

\begin{claimsb}
\label{claim:five}
$\tlambda_2(w,z) \cdot \BK{\AB{\hw}{z}}{\AK{\hw}{k}} \cdot \tlambda_2(w,z)^{-1} = \AK{\hw}{\BK{z}{k}}$
for $w \in (S_A^{\pm})^{\ast}$, $z \in \Z^n$ and $k \in \Gamma_n$.
\end{claimsb}
\begin{proof}[Proof of claim]
This can be proved by induction on the length of $w$ just like Claim \ref{claim:three}.  The details are left
to the reader.
\end{proof}

The next claim extends Claim \ref{claim:one} to show that $\tlambda_2$ satisfies a version of condition TB1.

\begin{claimsb}
\label{claim:six}
$\tlambda_2(w,z z') = \tlambda_2(w,z) \cdot \BK{\AB{\hw}{z}}{\tlambda_2(w,z')}$
for $w \in (S_A^{\pm 1})^{\ast}$ and $z,z' \in \Z^n$.
\end{claimsb}
\begin{proof}[Proof of claim]
The proof is by induction on the length of $w$.  For $w$ of length $0$, this is trivial, and
for $w$ of length $1$, it holds by Claim \ref{claim:one}.  Now assume that $w$ has length at least $2$
and that the desired formula holds whenever $w$ has smaller length.  Write $w = w_1 w_2$, where
$w_1$ and $w_2$ are shorter words than $w$.  Applying Claim \ref{claim:four} and our inductive hypothesis,
we see that
\begin{align}
\tlambda_2(w,z z') &= \AK{\hw_1}{\tlambda_2(w_2,z z')} \cdot \tlambda_2(w_1, \AB{\hw_2}{z} \AB{\hw_2}{z'}) \notag\\
&= \AK{\hw_1}{\tlambda_2(w_2,z) \cdot \BK{\AB{\hw_2}{z}}{\tlambda_2(w_2,z')}} \cdot
\tlambda_2(w_1,\AB{\hw_2}{z}) \cdot \BK{\AB{\hw_1}{z}}{\tlambda_2(w_1,\AB{\hw_2}{z'})}. \label{eqn:lefthandside}
\end{align}
Also, Claim \ref{claim:four} implies that $\tlambda_2(w,z) \cdot \BK{\AB{\hw}{z}}{\tlambda_2(w,z')}$ equals
\begin{equation}
\label{eqn:righthandside}
\AK{\hw_1}{\tlambda_2(w_2,z)} \cdot \tlambda_2(w_1,\AB{\hw_2}{z}) \cdot
\BK{\AB{\widehat{w_1 w_2}}{z}}{\AK{\hw_1}{\tlambda_2(w_2,z')} \cdot \tlambda_2(w_1,\AB{\hw_2}{z'})}.
\end{equation}
Our goal is to prove that \eqref{eqn:lefthandside} equals \eqref{eqn:righthandside}.  Manipulating this,
we see that our goal is equivalent to showing that
\[\tlambda_2(w_1,\AB{\hw_2}{z}) \cdot \BK{\AB{\widehat{w_1 w_2}}{z}}{\AK{\hw_1}{\tlambda_2(w_2,z')}} \cdot \tlambda_2(w_1,\AB{\hw_2}{z})^{-1}
=
\AK{\hw_1}{\BK{\AB{\hw_2}{z}}{\tlambda_2(w_2,z')}}.\]
This is an immediate consequence of Claim \ref{claim:five}.
\end{proof}

We finally prove the promised analogue of Claim \ref{claim:two}.

\begin{claimsb}
\label{claim:seven}
For $w,w' \in (S_A^{\pm 1})^{\ast}$ with $\hw = \hw' \in \Aut(F_n)$, we have $\tlambda_2(w,z) = \tlambda_2(w',z)$ for
$z \in \Z^n$.
\end{claimsb}
\begin{proof}[Proof of claim]
Recall that $\Aut(F_n) = \GroupPres{S_A}{R_A}$.
Define $R_A' = R_A \cup \Set{$s s^{-1}$}{$s \in S_A^{\pm 1}$} \subset (S_A^{\pm 1})^{\ast}$.
A computer calculation described below in Lemma \ref{lemma:computerrelationsaut} shows that
$\tlambda_2(r,\hs) = 1$ for $r \in R_A'$ and $s \in S_Z^{\pm 1}$.  Writing $z$ as a product of elements of
$S_Z^{\pm 1}$, we can use Claim \ref{claim:six} to show that $\tlambda_2(r,z) = 1$ for $r \in R_A'$.  The
proof now is identical to the proof of Claim \ref{claim:two}; the details are left to the reader.
\end{proof}

Claim \ref{claim:seven} implies that $\tlambda_2$ descends to a map $\lambda_2\co\Aut(F_n) \times \Z^n \rightarrow \Gamma_n$.
This map is a twisted bilinear map: Claim \ref{claim:six} implies that it satisfies condition TB1, 
Claim \ref{claim:four} implies that it satisfies condition TB2, and Claim \ref{claim:five} 
implies that it satisfies condition TB3.  As discussed at the beginning of the proof, $\lambda = \lambda_2$
is the twisted bilinear map whose existence we are trying to prove.
\end{proof}

\subsection{Proof of L-presentation}
\label{section:lpresentationproof}

We now prove Theorem \ref{theorem:preciselpres}.

\begin{proof}[{Proof of Theorem \ref{theorem:preciselpres}}]
Let $\Gamma_n = \LPres{S_K}{R_K^0}{\phi(S_Q^{\pm1})}$
be the group with the presentation described in Theorem \ref{theorem:preciselpres}.  
We map each generator of $\Gamma_n$ to the generator of $\BKerIA{n}{1}$ with the same name.
Lemma~\ref{lemma:computerpenultimatecalc} below checks that the basic relations $R_K^0$ are true in $\BKerIA{n}{1}$; it then follows from the naturality from Lemmas~\ref{lemma:action1} and~\ref{lemma:action2} that the extended relations of $\Gamma_n$ are also true in $\BKerIA{n}{1}$.
Therefore we have defined a homomorphism $\Psi\co\Gamma_n \rightarrow \BKerIA{n}{1}$.
Since our generating set from Theorem~\ref{maintheorem:generators} is in the image of $\Psi$, we know $\Psi$ is a surjection;
our goal is to show that $\Psi$ is an isomorphism.  

For $f \in \Aut(F_n)$, let $\alpha_f\co\Gamma_n \rightarrow \Gamma_n$ be the homomorphism
given by Lemma \ref{lemma:action1}.  Also, for $z \in \Z^n$, let $\beta_z\co\Gamma_n \rightarrow \Gamma_n$
be the homomorphism given by Lemma \ref{lemma:action2}.  Finally, let
$\lambda\co \Aut(F_n) \times \Z^n \rightarrow \Gamma_n$ be the twisted bilinear map given
by Lemma \ref{lemma:twistedbilinear}.  Plugging this data
into Theorem \ref{theorem:twistedextend}, we obtain a short exact sequence
\[1 \longrightarrow \Gamma_n \longrightarrow \Delta_n \stackrel{\rho}{\longrightarrow} \zasdr \longrightarrow 1\]
together with homomorphisms $\iota_1\co\Aut(F_n) \rightarrow \Delta_n$ and
$\iota_2\co\Z^n \rightarrow \Delta_n$ such that $\rho \circ \iota_1 = \text{id}$
and $\rho \circ \iota_2 = \text{id}$.  The naturality properties of the data
in Lemmas \ref{lemma:action1}, \ref{lemma:action2}, and \ref{lemma:twistedbilinear}
imply that this short exact sequence fits into a commutative diagram
\begin{equation}
\label{eqn:maindiagramused}
\begin{CD}
1  @>>> \Gamma_n        @>>> \Delta_n       @>>> \zasd @>>> 1  \\
@.      @VV{\Psi}V               @VV{\Phi}V              @VV{=}V                     @. \\
1  @>>> \BKerIA{n}{1} @>>> \AutFB{n}{1} @>>> \zasd @>>> 1.
\end{CD}
\end{equation}
By the five lemma, we see that to prove that $\Psi$ is an isomorphism, it
is enough to prove that $\Phi$ is an isomorphism.  We will do this
by constructing an explicit inverse homomorphism $\Phi^{-1}\co\AutFB{n}{1} \rightarrow \Delta_n$.

To do this, we first need some explicit elements of $\Delta_n$ and some relations
between those elements.  The needed elements are as follows.
\begin{compactitem}
\item We will identify the generating set 
\[S_K = \Set{$\kCon{y}{x_a}$, $\kCon{x_a}{y}$}{$x_a\in X$} \cup \Set{$\kMul{x_a^\alpha}{[y^\epsilon,x_b^\beta]}$}{$x_a,x_b\in X$, $x_a\neq x_b$, $\alpha,\beta,\epsilon\in\{1,-1,\}$}\]
for $\Gamma_n$ with its image in $\Delta_n$.  
\item For $\alpha \in \{1,-1\}$ and distinct $x_a, x_b \in X$, we define $\kMul{x_a^{\alpha}}{x_b} \in \Delta_n$
to equal $\iota_A(\Mul{x_a^{\alpha}}{x_b})$.  
\item For distinct $x_a, x_b \in X$, we define $\kSwap{a}{b}$ to equal $\iota_A(\Swap{a}{b})$.  
\item For $x_a \in X$, we define $\kInv{a}$ to equal $\iota_A(\Inv{a})$.
\item As in the proof of Lemma \ref{lemma:action2}, we will regard $\Z^n$ as being
generated by the set $\Set{$\Mul{x_a}{y}$}{$x_a \in X$}$, and for $x_a \in X$
we define $\kMul{x_a}{y}$ to equal $\iota_B(\Mul{x_a}{y})$.
\end{compactitem}
The needed relations are as follows.  That they hold is immediate from the construction of $\Delta_n$
in the proof of Theorem \ref{theorem:twistedextend}.
\begin{compactitem}
\item The relations $R_K^0$ from the L-presentation for $\Gamma_n$.
\item By construction, the group $\Delta_n$ contains subgroups $\Gamma_n \rtimes \Aut(F_n)$
and $\Gamma_n \rtimes \Z^n$.  Any relation which holds in $\Gamma_n \rtimes \Aut(F_n)$
or $\Gamma_n \rtimes \Z^n$ (which are generated by the evident elements) also holds in $\Delta_n$.
\item For $f \in \Aut(F_n)$ and $z \in \Z^n$, we have $\lambda(f,z) = f z f^{-1} \AB{f}{z}^{-1}$.  Here
$\AB{f}{z}$ comes from the action of $\Aut(F_n)$ on $\Z^n$ in the semidirect product $\zasd$.  Also,
$f \in \Aut(F_n)$ and $z \in \Z^n$ and $\AB{f}{z} \in \Z^n$ should be identified with their images in $\Delta_n$.
\end{compactitem}

Let $\AutFB{n}{1} = \Pres{S_C}{R_C}$ be the presentation given by
Theorem \ref{th:JensenWahl}, so 
\begin{align*}
S_C = &\Set{$\Mul{x_a^{\alpha}}{x_b}$}{$1 \leq a,b \leq n$ distinct, $\alpha \in \{1,-1\}$}
\cup \Set{$\Swap{a}{b}$}{$1 \leq a < b \leq n$}\\
&\cup \Set{$\Inv{a}$}{$1 \leq a \leq n$}
\cup \Set{$\Mul{x_a^\alpha}{y}$,$\Con{y}{x_a}$}{$1\leq a\leq n$, $\alpha\in\{1,-1\}$}.
\end{align*}
We define a set map $\widetilde{\Phi}^{-1}\co S_C \rightarrow \Delta_n$ as follows.  First, most of the
elements in $S_C$ have evident analogues in $\Delta_n$, so we define
\begin{align*}
&\widetilde{\Phi}^{-1}(\Mul{x_a^{\alpha}}{x_b}) = \kMul{x_a^{\alpha}}{x_b} \quad \text{and} \quad
\widetilde{\Phi}^{-1}(\Swap{a}{b}) = \kSwap{a}{b} \quad \text{and} \quad
\widetilde{\Phi}^{-1}(\Inv{a}) = \kInv{a}\\
& \quad \text{and} \quad
\widetilde{\Phi}^{-1}(\Mul{x_a}{y}) = \kMul{x_a}{y} \quad \text{and} \quad
\widetilde{\Phi}^{-1}(\Con{y}{x_a}) = \kCon{y}{x_a}.
\end{align*}
The only remaining element of $S_C$ is $\Mul{x_a^{-1}}{y}$, and we define
\[\widetilde{\Phi}^{-1}(\Mul{x_a^{-1}}{y}) = \kCon{x_a}{y} \kMul{x_a}{y}^{-1}.\]
The map $\widetilde{\Phi}^{-1}$ extends to a homomorphism $\widetilde{\Phi}^{-1}\co F(S_C) \rightarrow \Delta_n$.
Computer calculations described in Lemma \ref{lemma:computerfinalcalc} below show that $\widetilde{\Phi}^{-1}(r) = 1$
for $r \in R_C$, so $\widetilde{\Phi}^{-1}$ descends to a homomorphism $\Phi^{-1}\co \AutFB{n}{1} \rightarrow \Delta_n$.
Examining its effect on generators, we see that $\Phi^{-1}$ is the desired inverse to $\Phi$, and the
proof is complete.
\end{proof}

\section{Computer calculations}
\label{section:computer}

This section discusses the computer calculations used in the previous section.  The
preliminary section \S \ref{section:computerdisc} discusses the basic framework we use.  The
actual computations are in \S \ref{section:computercalc}.

\subsection{Framework for calculations}
\label{section:computerdisc}

We model $\Aut(F_{n+1})$ using GAP, a software algebra system available for free at \url{http://www.gap-system.org/}.
We encourage our readers to experiment with the included functions, and to look at the code that performs the verifications below.
We use the same framework that the authors used in~\cite{DayPutmanH2}, so we quote part of our explanation of the framework from there.
From~\cite{DayPutmanH2}:
\begin{quote}
We use GAP's built-in functionality to model $F_n$ as a free group on the eight generators \verb+xa+, \verb+xb+, \verb+xc+, \verb+xd+, \verb+xe+, \verb+xf+, \verb+xg+, and \verb+y+.
Since our computations never involve more than $8$ variables, computations in this group suffice to show that our computations hold in general.
\end{quote}


We found it more convenient to model the free groups $F(S_A)$, $F(S_Q)$, and $F(S_K)$ without using the built-in free group functionality.
Instead we model the generators using lists and program the basic free group operations directly.
Continuing from~\cite{DayPutmanH2}:
\begin{quote}
For example, we model the generator $\Mul{x_a}{x_b}$ as the list \verb+["M",xa,xb]+, $\Con{y}{x_a}$ as \verb+["C",y,xa]+, and $\Mulcomm{x_a^{-1}}{y}{x_c}$ as \verb+["Mc",xa^-1,y,xc]+.
We model $P_{a,b}$ as \verb+["P",xa,xb]+ and $I_a$ as \verb+["I",xa]+.
The examples should make clear: the first entry in the list is a string key \verb+"M"+, \verb+"C"+, \verb+"Mc"+, \verb+"P"+, or \verb+"I"+, indicating whether the list represents a transvection, conjugation move, commutator transvection, swap or inversion.
The parameters given as subscripts in the generator are then the remaining elements of the list, in the same order.
\end{quote}

We model words in any of the free groups $F(S_A)$, $F(S_Q)$, and $F(S_K)$ as lists of generators.
Continuing from~\cite{DayPutmanH2}:
\begin{quote}
We model inverses of generators as follows:
the inverse of \verb+["M",xa,xb]+ is \verb+["M",xa,xb^-1]+ and the inverse of \verb+["C",xa,xb]+ is  \verb+["C",xa,xb^-1]+, but the inverse of \verb+["Mc",xa,xb,xc]+ is \verb+["Mc",xa,xc,xb]+.
Swaps and inversions are their own inverses.
Technically, this means that \ldots we model structures where the order relations for swaps and inversions and the relation \ldots for inverting commutator transvections are built in.
This is not a problem because our verifications always show that certain formulas are trivial modulo our relations \ldots.
\end{quote}

In particular, the inverse of \verb+["Mc",xa,y,xb]+ is modeled as \verb+["Mc",xa,xb,y]+.
Continuing from~\cite{DayPutmanH2}:
\begin{quote}
The empty word \verb+[]+ represents the trivial element.
We wrote several functions \ldots that perform common tasks on words.
The function \verb+pw+ takes any number of words (reduced or not) as arguments and returns the freely reduced product of those  words in the given order, as a single word.
The function \verb+iw+ inverts its input word and the function \verb+cyw+ cyclically permutes its input word.

\ldots 
The function \verb+applyrels+ is particularly useful, because it inserts multiple relations into a word.
It takes in two inputs: a starting word and a list of words with placement indicators.
The function recursively inserts the first word from the list in the starting word at the given position, reduces the word, and then calls itself with the new word as the starting word and with the same list of insertions, with the first dropped.
\end{quote}

Most of the verifications amount to showing that some formula can be expressed as a product of conjugates of images of relations under the substitution rules.
We model the substitution rule function $\phi$ using a function named \verb+phi+.
This takes a word in $\fm{(S_Q^{\pm1})}$ as its first input and a word from $F(S_K)$ as its second input and applies to the second word the composition of substitution rules given by the first one.
We use a function \verb+krel+ to encode the basic relations $R_K^0$ from Theorem~\ref{theorem:preciselpres}.
Given a number $n$ and a list of basis elements (or inverse basis elements) from $F_{n+1}$, \verb+krel+ returns the $n$th relation from Table~\ref{table:basicrelsbkerian}, with the supplied basis elements as subscripts on the $S_K$-generators.
If the parameters are inconsistent, it returns the empty word.
We define a function \verb+psi+ that encodes the action of $\Aut(F_n)$ on $\Z^n$ from Section~\ref{section:relatingkernels}.
We also defines a function \verb+lambda+ that computes the definition of $\tlambda_2$ above; in the special case that its first input is in $S_A^{\pm1}$, it computes $\tlambda_1$ as well (this is also true of $\tlambda_2$ by definition).

The functions described here and the checklists for the computations described below are all given in the file \verb+BirmanIA.g+, which we make available with this paper.
Each of the following lemmas refers to lists of outputs in \verb+BirmanIA.g+.
To check the validity of a given lemma, one needs to read the code that generates the list, evaluate the code, and make sure the output is correct (usually the desired output is a list of copies of the trivial word).
We have provided a list \verb+BirmanIAchecklist+ that gives the output of all the verifications in the paper.



\subsection{The actual calculations}
\label{section:computercalc}

In addition to the relations from Table~\ref{table:basicrelsbkerian}, we use some derived relations for convenience.
These are output by a function \verb+exkrel+ and we do not list them here (they can be found by inspecting the code and the outputs from that function).
What matters is that these relations always follow from the relations in the presentation.
\begin{lemma}
All the relations output by \verb+exkrel+ are true in $\Gamma_n$.
\end{lemma}

\begin{proof}
The source for the list \verb+exkrellist+ contains a reduction of one instance of each output of \verb+exkrel+ to the trivial word using only outputs from \verb+krel+, images of outputs from \verb+krel+ under the action of \verb+phi+, and previously verified relations from \verb+exkrel+.
Each of the entries in the list evaluates to the trivial word, so the reductions are correct.
\end{proof}

There are a few places where we verify identities that are homomorphisms on both sides.
To verify these most efficiently, we use generating sets for $\Gamma_n$ that are smaller than $S_K$.
\begin{lemma}\label{le:Gammagenset}
Suppose $S$ is a set containing all the $\{\kCon{x_a}{y}\}_a$ and $\{\kCon{y}{x_a}\}_a$, and for each choice of $a,b$, suppose $S$ contains at least one of the eight elements $\{\kMulcomm{x_a^\alpha}{y^\epsilon}{x_b^\beta}\}_{\alpha,\beta,\epsilon}$.
Then $S$ is a generating set for $\Gamma_n$.
\end{lemma}
\begin{proof}
Suppose $S$ contains $\kMulcomm{x_a^\alpha}{y^\epsilon}{x_b^\beta}$ and all the conjugation moves above.
Then: 
 we can use R4 to express $\kMulcomm{x_a^\alpha}{y^\epsilon}{x_b^{-\beta}}$ in terms of elements of $S$,
we can use R5 to express $\kMulcomm{x_a^\alpha}{y^{-\epsilon}}{x_b^\beta}$ in terms of elements of $S$, and we can use R7 to express $\kMulcomm{x_a^{-\alpha}}{y^\epsilon}{x_b^\beta}$ in terms of elements of $S$.
Applying these relations repeatedly allows us to get all of the eight commutator transvections involving $x_a$ and $x_b$ from one of them.
\end{proof}

\begin{lemma}
\label{lemma:computerphiproperty}
Let $\Psi\co F(S_K) \rightarrow \BKerIA{n}{1}$ be the natural surjection.  Then
regarding $S_Q$ as a subset of $\AutFB{n}{1}$, we have
\[\Psi(\phi(s)(t)) = s \Psi(t) s^{-1} \quad \quad (s \in S_Q, t \in S_K).\]
\end{lemma}
\begin{proof}
First we note that this is clearly true, by definition, if $s$ is a swap or inversion.
Further, if we verify this for $s=\kMul{x_a}{x_b}$, then it follows for $s=\kMul{x_a^{-1}}{x_b}$ by conjugating the entire expression by an inversion.
So it is enough to verify it for $s$  of the form $\kMul{x_a}{x_b}$ or $\kMul{x_a}{y}$.
In the code generating the list \verb+phiconjugationlist+, we check this equation for both such choices of $s$, and for all possible configurations of generator $t$ with respect to the choice of $s$.
\end{proof}

\begin{lemma}
\label{lemma:computeraction11}
The map
$\zeta\co (S_A^{\pm 1})^{\ast} \rightarrow \End(\Gamma_n)$
induced by $\phi\co S_Q^{\pm 1} \rightarrow \End(F(S_K))$
satisfies
\[\zeta(s)(\zeta(s^{-1})(t)) = t\]
for all $s\in S_A^{\pm1}$ and all $t\in S_K$.
\end{lemma}

\begin{proof}
In fact, it is enough to show this for $t$ in a generating set for $\Gamma_n$, since then $\zeta(s)\circ\zeta(s^{-1})$ is the identity endomorphism of $\Gamma_n$.
If $s$ is a swap or inversion, then the lemma follows immediately from the definition.
So we verify that $\phi(s)(\phi(s^{-1})(t))=t$ (up to relations of $\Gamma_n)$ for $s$ of the form $\kMul{x_a}{x_b}$, and for enough choices of $t$ to give a generating set for $\Gamma_n$ (using Lemma~\ref{le:Gammagenset}.
This computation appears in the code generating the list \verb+phiAinverselist+.
For $s$ of the form $\kMul{x_a^{-1}}{x_b}$, our computations for $s=\kMul{x_a}{x_b}$ suffice, after substituting $x_a$ for $x_a^{-1}$ in each computation.
\end{proof}

\begin{lemma}
\label{lemma:computeraction12}
The map
$\eta\co F(S_A) \rightarrow \Aut(\Gamma_n)$
induced by $\phi\co S_Q^{\pm 1} \rightarrow \End(F(S_K))$
satisfies
\[\eta(r)(t)=t\]
for all $t\in S_K$ and for every relation $r\in R_A$ from Nielsen's presentation for $\Aut(F_n)=\langle  S_A|R_A\rangle$.
\end{lemma}

\begin{proof}
It is enough to check that the equation
\begin{equation}
\label{eq:phirelation}
\phi(r)(t)=t
\end{equation}
holds in $\Gamma_n$ for every relation $r$ from Nielsen's presentation, for choices of $t$ ranging through a generating set for $\Gamma_n$.

Equation~\eqref{eq:phirelation} works automatically for $r$ a relation of type N1, since the action $\phi$ is defined for swaps and inversions using the natural action on $\{x_1^{\pm1},\dotsc,x_n^{\pm1}\}$, and these relations hold for that action.

Equation~\eqref{eq:phirelation} also works automatically for $r$ a relation of type N2.
In this case, the equation says that $s\mapsto\phi(s)$, for $s$ a transvection, is equivariant with respect to the action of swaps and inversions.
This is apparent from the definition of $\phi$: the definition does not refer to specific elements $x_i$, but instead treats configurations the same way based on coincidences between them.

For the other cases, we use computations given in the source code for the lists \verb+phiN3list+, \verb+phiN4list+, and \verb+phiN5list+.
In each list we select a relation $r$ and reduce $\phi(r)(t)t^{-1}$ to $1$ in $\Gamma_n$, for choices of $t$ constituting a generating set by Lemma~\ref{le:Gammagenset}.
In each list we exploit natural symmetries of the equation to reduce the number of cases considered.
\end{proof}

\begin{lemma}
\label{lemma:computeraction21}
The map 
$\zeta\co(S_Z^{\pm 1})^{\ast} \rightarrow \End(\Gamma_n)$
induced by $\phi\co S_Q^{\pm 1} \rightarrow \End(F(S_K))$
satisfies
\[\zeta(s)(\zeta(s^{-1})(t)) = t\]
for all $s \in S_Z^{\pm 1}$ and $t \in S_K$.
\end{lemma}

\begin{proof}
This is like the proof of Lemma~\ref{lemma:computeraction11}, but simpler.
We verify that $\phi(s)(\phi(s^{-1})(t)=t$ in $\Gamma_n$ for $s$ of the form $\kMul{x_a}{y}$,  for  $t$ ranging over a generating set for $\Gamma_n$.
This computation is in the code generating the list \verb+phiZinverselist+.
\end{proof}

\begin{lemma}
\label{lemma:computeraction22}
The map $\eta\co F(S_Z) \rightarrow \Aut(\Gamma_n)$
induced by $\phi\co S_Q^{\pm 1} \rightarrow \End(F(S_K))$
satisfies
\[\eta(r)(t)=t,\]
whenever $r$ is a basic commutator of generators from $S_Z$ and $t\in S_K$.
\end{lemma}

\begin{proof}
The computations showing this appear in \verb+phiznlist+.
\end{proof}

\begin{lemma}
\label{lemma:computerrelationsz}
For $r$ of the form $ss^{-1}$ for $s\in S_Z$, or $[s,t]$ for $s,t\in S_Z$, 
we have 
\[\tlambda_1(f,r)=1\]
in $\Gamma_n$ for any $f\in S_A^{\pm1}$.
\end{lemma}

\begin{proof}
The meaning here is that we must expand $\tlambda_1(f,r)$ according the definition without simplifying $r$ (not even cancelling inverse pairs), and then verify that the expression we get is a relation in $\Gamma_n$.
To check the lemma, it is enough to verify that the function \verb+lambda+ returns relations in $\Gamma_n$ when the first input is a generator or inverse generator and the second input is commutator or the product of a generator and its inverse.
The computations for this lemma this appear in \verb+lambda2ndinverselist+ and \verb+lambda2ndrellist+.
\end{proof}

\begin{lemma}
\label{lemma:computertb3}
We have
\[\lambda_1(f,z) \cdot \BK{\AB{\hf}{z}}{\AK{\hf}{t}} \cdot \lambda_1(f,z)^{-1} = \AK{\hf}{\BK{z}{t}}\]
for $f \in S_A$, $z \in S_Z$, and $t \in \Gamma_n$.
\end{lemma}

\begin{proof}
The function \verb+psi+ takes as input a word $a$ in $S_A$ and a word $b$ in $S_Z$ and returns a word in $S_Z$ representing $\AB{a}{b}$.
Since the actions $\alpha$ and $\beta$ are given by $\phi$, we may rewrite the expression we are trying to prove as
\[
\lambda_1(f,z)\cdot \phi(\psi(f)(z))\circ\phi(f)(t) \cdot \lambda_1(f,z)^{-1} = \phi(f)\circ \phi(z)(t).
\]

We note that this equation is an automorphism of $\Gamma_n$ on both sides, so it is enough to verify it for $t$ in a generating set.
Computations checking this identity for all choices of $f$, $z$ and $t$ appear in the code generating \verb+tb3list+.
\end{proof}

\begin{lemma}
\label{lemma:computerrelationsaut}
Suppose $r\in (S_A)^{\ast}$ is one of Nielsen's relations for $\Aut(F_n)$ or is a product $ff^{-1}$ for some $f\in S_A^{\pm1}$.
Then for any $z\in \Z^n$, we have
\[\lambda_2(r,z)=1\]
in $\Gamma_n$.
\end{lemma}

\begin{proof}
It is enough to show this for generators of $\Z^n$.
We show this using the function \verb+lambda+ that encodes the definition of $\lambda_2$.
The code checking these identities generates the lists \verb+lambda1stinversecheck+, \verb+lambdaN1list+, \verb+lambdaN2list+, \verb+lambdaN3list+, \verb+lambdaN4list+, and \verb+lambdaN5list+.
\end{proof}

\begin{lemma}
\label{lemma:computerpenultimatecalc}
The basic relations $R_K^0$ of the L-presentation for $\Gamma_n$ are true when interpreted as identities in $\BKerIA{n}{1}$.
\end{lemma}

\begin{proof}
This is verified in the list \verb+verifyGammarellist+, which uses the function \verb+krel+ to generate the relations.
All generic and non-generic forms of the relations are checked separately.
\end{proof}

\begin{lemma}
\label{lemma:computerfinalcalc}
The relations from Jensen--Wahl's presentation for $\AutFB{n}{1}$ map to relations of $\Delta_n$ under the map $\widetilde{\Phi}^{-1}$, so the map as defined on generators extends to a well defined homomorphism $\Phi^{-1}\co \AutFB{n}{1}\to \Delta_n$.
\end{lemma}

\begin{proof}
This is verified in the list \verb+JWfromDeltalist+.
\end{proof}

\noindent
\begin{tabular*}{\linewidth}[t]{@{}p{\widthof{Department of Mathematical Sciences, 309 SCEN}+0.50in}@{}p{\linewidth - \widthof{Department of Mathematical Sciences, 309 SCEN} - 0.50in}@{}}
{\raggedright
Matthew Day\\
Department of Mathematical Sciences, 309 SCEN\\
University of Arkansas\\
Fayetteville, AR 72701\\
E-mail: {\tt matthewd@uark.edu}}
&
{\raggedright
Andrew Putman\\
Department of Mathematics\\
Rice University, MS 136 \\
6100 Main St.\\
Houston, TX 77005\\
E-mail: {\tt andyp@math.rice.edu}}
\end{tabular*}

\end{document}